\documentclass[12pt]{amsart}
\usepackage{amssymb}
\usepackage{amsmath,latexsym}
\usepackage{amsfonts}
\usepackage{amsthm}
\usepackage{eucal}
\allowdisplaybreaks[4]
\newcommand{\R}{\mathbb{R}}

\newcommand{\vol}{{\rm vol}}

\numberwithin{equation}{section}

\newtheorem{thm}{Theorem}[section]
\newtheorem{defn}[thm]{Definition}
\newtheorem{lem}[thm]{Lemma}
\newtheorem{prop}[thm]{Proposition}
\newtheorem{cor}[thm]{Corollary}

\newtheorem{remark}[thm]{Remark}

\begin{document}

\title[]{Perelman's $\lambda$-functional on manifolds with conical singularities}
\author{Xianzhe Dai}
\address{
Department of Mathematics, East China Normal University, Shanghai, China, and
University of Californai, Santa Barbara
CA93106,
USA}
\email{dai@math.ucsb.edu}

\author{Changliang Wang}
\address{Department of Mathematics and Statistics, McMaster University, Hamilton, Ontario, Canada}
\email{wangc114@math.mcmaster.ca}
\date{}

\begin{abstract}
In this paper, we prove that on a compact manifold with isolated conical singularity the spectrum of the Schr\"odinger operator $-4\Delta+R$ consists of discrete eigenvalues with finite multiplicities, if the scalar curvature $R$ satisfies a certain condition near the singularity. Moreover, we obtain an asymptotic behavior for eigenfunctions near the singularity. As a consequence of these spectral properties, we extend the theory of the Perelman's $\lambda$-functional on smooth compact manifolds to compact manifolds with isolated conical singularities.
\end{abstract}

\maketitle

\begin{section}{Introduction}
\noindent Let $(M,g)$ be a compact Riemannian manifold without boundary. We recall some Riemannian functionals introduced by G. Perelman to study Ricci flows\cite{Per02}. The $\mathcal{F}$-functional is defined by
\begin{equation}\label{F-functional}
\mathcal{F}(g,f)=\int_{M}(R_{g}+|\nabla f|)e^{-f}d\vol_{g},
\end{equation}
where $R_{g}$ is the scalar curvature of the metric $g$, and $f$ is a smooth function on $M$. Let $u=e^{-\frac{f}{2}}$, then the $\mathcal{F}$-functional becomes
\begin{equation}\label{F-functional2}
\mathcal{F}(g,u)=\int_{M}(4|\nabla u|^{2}+R_{g}u^{2})d\vol_{g}.
\end{equation}
The Perelman's $\lambda$-functional is defined by
\begin{equation}\label{lambda-functional}
\lambda(g)=\inf\{\mathcal{F}(g,u) \ \ | \ \ \int_{M}u^{2}d\vol_{g}=1\}.
\end{equation}

Clearly, from $(\ref{lambda-functional})$ and $(\ref{F-functional2})$,  $\lambda(g)$ is the smallest eigenvalue of the Schr\"odinger operator $-4\Delta_{g}+R_{g}$. Since, in this case,  the smallest eigenvalue is simple, $\lambda(g)$ and the corresponding normalized positive eigenfunction depend smoothly in $g$ (see e.g. \cite{DWW05}). Let $g(t)$, for $t\in(-\epsilon,\epsilon)$, be a smooth family of metrics with $g(0)=g_{0}$ and $\frac{d}{dt}g(t)|_{t=0}=h$. The first variation formula of the $\lambda$-functional is given by\cite{Per02}
\begin{equation}\label{FristVariationalLambdaFunctional}
\frac{d}{dt}\lambda(g(t))|_{t=0}=\int_{M}\langle-Ric_{g_{0}}-\nabla^{2}_{g_{0}}f,h\rangle_{g_{0}} e^{-f}d\vol_{g_{0}}.
\end{equation}
From the formula $(\ref{FristVariationalLambdaFunctional})$, we can see that the critical points of the $\lambda$-functional are the steady Ricci solitons, which are actually the Ricci flat metrics because the manifold is compact. Moreover, the gradient flow of the $\lambda$-functional is the Ricci flow up to diffeomorphisms, and thus the $\lambda$-functional is nondecreasing along the Ricci flow.

At a critical point of the $\lambda$-functional, i.e. a Ricci flat metric $g_{0}$, the second variational formula of the $\lambda$-functional is
\begin{equation}\label{SecondVariationalLambdaFunctional}
\frac{d^{2}}{dt^{2}}\lambda(g(t))|_{t=0}=e^{-f}\int_{M}\langle -\frac{1}{2}\Delta_{L,g_{0}}h+\delta_{g_{0}}^{*}\delta_{g_{0}} h+\frac{1}{2}\nabla^{2}_{g_{0}}(\nu_{h}),h\rangle_{g_{0}} d\vol_{g_{0}},
\end{equation}
where $\Delta_{g_{0}}\nu_{h}=-\delta_{g_{0}}(\delta_{g_{0}} h)$, $\Delta_{L,g_{0}}$ is the Lichnerowicz Laplace with respect to the metric $g_{0}$, $\delta_{g_{0}}$ is the divergence, and $\delta_{g_{0}}^{*}$ is the adjoint of $\delta_{g_{0}}$. This formula is the key to the study of stability of Ricci flat metrics and Ricci flow (see, e.g. \cite{Ses06}, \cite{Has12}, and \cite{HM14}).

Perelman also introduced and studied $W$-functional and $\mu$-entropy on smooth compact manifolds in \cite{Per02}. The second variation formula of corresponding $\nu$-entropy was derived in \cite{CHI04} and \cite{CZ12}. The $W$-functional on noncompact manifolds was studied in \cite{Zha12}.

In this paper we develop Perelman's $\lambda$-functional on a class of singular manifolds, the manifolds with isolated conical singularities. Recent advances have demonstrated the importance of singular manifolds in that they are not only important in themselves but also provide a powerful tool for the study of smooth metrics, see, \cite{CDS15} and \cite{Tian15}. Riemannian manifolds with conical singularities also naturally appear as Gromov-Hausdorff limits of smooth manifolds, and as singularities of Ricci flow. These motivate us to study Riemannian manifolds with conical singularities, in particular, the theory of Perelman's $\lambda$-functional on these manifolds.

As we have seen that the $\lambda$-functional on a compact smooth manifold is the same as the smallest eigenvalue of the Schr\"odinger operator $-4\Delta+R$. We will take this spectral viewpoint in developing Perelman's $\lambda$-functional on compact manifolds with isolated conical singularities. Spectral theory for Laplacians on manifolds with conical singularity, initiated by Cheeger\cite{Che}, has a long history and is relatively well understood. The complication here comes from the singular potential caused by the conical singularity. Therefore, we first extend the spectral theory for the Schr\"odinger operator $-4\Delta+R$ to compact Riemannian manifolds with isolated conical singularities.

Roughly speaking, by a compact Riemannian manifold with isolated conical singularities we mean a singular manifold $(M, g)$ whose singular set $S$ consists of finite many points and its regular part $(M\setminus S, g)$ is a smooth Riemannian manifold. Moreover, near the singularities, the metric is asymptotic to a (finite) metric cone $C_{(0, 1]}(N)$ where $N$ is a compact smooth Riemannian manifold with metric $h_0$ which will be called a cross section (see \S 1 for the precise definition).
Our first main result establishes the spectral theory for the Schr\"odinger operator with singular potential $-4\Delta+R$ on the manifolds with isolated conical singularities.

\begin{thm}\label{MainResult1}
Let $(M^{n}, g)$ ($n\geq 3$) be a compact Riemannian manifold with isolated conical singularity. If the scalar curvature of the cross sections at the conical singularities $R_{h_{0}}>(n-2)$ on $N$, then the operator $-4\Delta_{g}+R_{g}$ with domain $C^{\infty}_{0}(M\setminus S)$ is semi-bounded, and the spectrum of its Friedrichs extension consists of discrete eigenvalues with finite multiplicities
$\lambda_{1}<\lambda_{2}\leq\lambda_{3}\leq\cdots,$ and $\lambda_{k}\rightarrow+\infty$, as $k\rightarrow+\infty$. The corresponding eigenfunctions form a complete basis of $L^{2}(M)$
\end{thm}

\begin{remark}
	For $2$-dimensional manifolds with conical singularities, the curvature is flat for the model cone and no assumption is needed in this case. The analysis and result essentially follow from Cheeger \cite{Che}.
\end{remark}

\begin{remark}
	Similarly the result extends to the Schr\"odinger operator $-\Delta_{g}+ c R_{g}$ for constant $c$, where the scalar curvature condition becomes $R_{h_{0}}> \frac{(n-1)[4c(n-2)-(n-3)]-1}{4c}$. For example, for the well known conformal Laplacian $-\Delta_{g}+ \frac{n-2}{4(n-1)} R_{g}$, the scalar curvature condition becomes $R_{h_{0}}> 0$.
\end{remark}

Thus on a compact manifold with isolated conical singularities and dimension $\geq 3$, we need the geometric condition $R_{h_{0}}>(n-2)$ to guarantee that the Friedrichs extension of the Schr\"odinger operator $-4\Delta_{g}+R_{g}$ exists and its spectrum behaves nicely as in the smooth case. This may be somewhat surprising and we make some further comments below.

Let $C(N,h_{0})=(\mathbb{R}_{+}\times N,g=dr^{2}+r^{2}h_{0})$ be the Riemannian cone over a compact (n-1)-dimensional smooth Riemannian manifold $(N^{n-1},h_{0})$. Then the Laplcian and the scalar curvature on the cone are given by
\begin{eqnarray}
\Delta_{g} &=& \partial^{2}_{r}+\frac{n-1}{r}\partial_{r}+\frac{1}{r^{2}}\Delta_{h_{0}},\label{Lapace_on_cone}\\
R_{g} &=& \frac{1}{r^{2}}[R_{h_{0}}-(n-1)(n-2)].\label{Scale_curvature_on_cone}
\end{eqnarray}

Thus the operator $-4\Delta+R$ is a Schr\"{o}dinger operator with singular potential $R$, which behaves like $O(\frac{1}{r^{2}})$ as $r\rightarrow0$. This type of operators have been studied in several literatures, see e.g. \cite{BS87} and \cite{RS2}. The simplest example of this type of operators is mentioned in \cite{BS87}. That is, the unbounded operator $L_{a}=-\frac{d^{2}}{dx^{2}}+\frac{a}{x^{2}}$ in $L^{2}(\mathbb{R}_{+})$ with the domain $D(L_{a})=C^{\infty}_{0}(\mathbb{R}_{+})$, where $a\in\mathbb{R}$ is a constant and $\mathbb{R}_{+}=(0,+\infty)$. By using Hardy's inequality, one can easily see that the operator $L_{a}$ is non-negative if and only if $a\geq-\frac{1}{4}$.

Moreover, in \cite{RS2}, M. Read and B. Simon studied the Schr\"{o}dinger operator with a spherically symmetric potential: $-\Delta+V(r)$ on $\R^{n}$, where $r=(\sum^{n}_{i=1}x^{2}_{i})^{\frac{1}{2}}$. They proved that if $V(r)+\frac{(n-1)(n-3)}{4}\frac{1}{r^{2}}\geq\frac{3}{4r^{2}}$,
then $-\Delta+V(r)$ is essentially self-adjoint on $C^{\infty}_{0}(\R^{n}\setminus\{0\})$, and if $V(r)$ satisfies $0\leq V(r)+\frac{(n-1)(n-3)}{4}\frac{1}{r^{2}}\leq\frac{c}{r^{2}}, \ \ \ c<\frac{3}{4}$, then $-\Delta+V(r)$ is not essentially self-adjoint on $C^{\infty}_{0}(\R^{n}\setminus\{0\})$ (see Theorem \uppercase\expandafter{\romannumeral10}.11 in \cite{RS2}).

We should also point out that in \cite{BP03}, B. Botvinnik and S. Preston proved that the spectrum of the conformal Laplacian on a compact Riemannian manifold with isolated tame conical singularities consists of discrete eigenvalues with finite multiplicities. The conformal Laplacian $-\Delta+\frac{n-2}{4(n-1)}R$ is a similar singular Schr\"{o}dinger operator. A tame conical singularity is given as a cone over a product of the standard spheres. Therefore, the scalar curvature of the cross section of a tame conical singularity clearly satisfies the condition in Theorem $\ref{MainResult1}$ (see the Remark after), and Theorem \ref{MainResult1} gives more general result in this case.

The general idea of the proof of Theorem $\ref{MainResult1}$ is similar to the one in \cite{BP03}. We use certain weighted Sobolev spaces that can be compactly embedded in $L^{2}(M)$. And by doing some estimates, we show that the operator $-4\Delta+R$ is semi-bounded and the domain of its self-adjoint Friedrichs extension is in a weighted Sobolev space. Then we can use the spectrum theorem for self-adjoint compact operators to obtain the property of the spectrum of the operator $-4\Delta+R$.

Theorem $\ref{MainResult1}$ enables us to define the $\lambda$-functional on manifolds with isolated conical singularities, as the smallest eigenvalue of $-4\Delta+R$. However, for deriving variational formulae of the the $\lambda$-functional, certain asymptotic behavior of eigenfunctions near singularities is necessary. This is our second main result in this paper.

\begin{thm}\label{MainResult2}
Let $(M^{n},g,p)$ be a compact Riemannian manifold with (for simplicity) a single conical point $p$ satisfying $R_{h_{0}}>(n-2)$ and
\begin{equation}\label{AsymptoticOfMetric2}
\begin{aligned}
&\ \ \ \ \  r^{i}|\nabla^{i+1}(h_{r}-h_{0})|\leq C_{i}<+\infty, \\
&\textit{for some constant} \ \ C_{i}, \ \ \textit{and each} \ \ 0\leq i\leq\frac{n}{2}+2,
\end{aligned}
\end{equation}
near $p$. Then eigenfunctions of $-4\Delta_{g}+R_{g}$ satisfy
\begin{equation}\label{AsymptoticOrder1}
|\nabla^{i}u|=o(r^{-\frac{n-2}{2}-i}), \ \ \ \text{as} \ \ r\rightarrow0,
\end{equation}
for $i=0$ and $1$.

Moreover, if the singularity is cone-like, eigenfunctions have asymptotic expansion at the conical singularity $p$ as
\begin{equation}\label{AsymptoticExpansion1}
u\sim \sum^{+\infty}_{j=1}\sum^{+\infty}_{l=0}\sum^{p_{j}}_{p=0}r^{s_{j}+l}(\ln r)^{p}u_{j,l,p},
\end{equation}
where $u_{j,l,p}\in C^{\infty}(N^{n-1})$, $p_{j}=0$ or $1$, and $s_{j}=-\frac{n-2}{2}\pm\frac{\sqrt{\mu_{j}-(n-2)}}{2}$, where $\mu_{j}$ are eigenvalues of $-\Delta_{h_{0}}+R_{h_{0}}$ on $N^{n-1}$.
\end{thm}

\begin{remark}
After we finish this paper, we notice the very recent paper of T. Ozuch \cite{Tri17}, in which he studies Perelman's functionals on cones from a different viewpoint. And reuslts of his paper partially overlap with some of the results in this paper. 
More interestingly he uses these results to study Ricci flows coming out of a cone.
\end{remark}

On a manifold with a cone-like singularity, a small neighborhood of the singularity is a finite exact cone over a compact smooth manifold. In this neighborhood, we can separate variables and explicitly solve the eigenfunction equations in term of eigenfunctions on the cross section of the cone and some hypergeometric functions. By using classical elliptic estimates and some estimates for the hypergeometric functions, we then obtain the asymptotic expansion $(\ref{AsymptoticExpansion1})$ of eigenfunctions on manifold with a cone-like singularity.

On a manifold with a conical singularity, we cannot do explicit calculations. Therefore, instead, we establish some estimates to obtain an asymptotic order near the singularity for eigenfunctions in $(\ref{AsymptoticOrder1})$. We first work on small finite cones, on which we can obtain some weighted Sobelov inequalities and weighted elliptic estimates by using scaling technique. The asymptotic condition $(\ref{AsymptoticOfMetric2})$ for the asymptotically conical metric implies weighted Sobelov norms and weighted $C^{k}$-norms with respect to exactly conical metric $dr^{2}+r^{2}h_{0}$ are equivalent to ones with respect to asymptotically conical metric $dr^{2}+r^{2}h_{r}$. Then these weighted Sobelov inequality and weighted elliptic estimates still hold on an asymptotic finite cone. This implies the asymptotic behavior of the eigenfunctions in $(\ref{AsymptoticOrder1})$ by using elliptic bootstrapping. And finally, we obtain variation formulae of $\lambda$-functional on compact manifolds with a single conical singularity with the help of the asymptotic information. Then, from the first variaiton formula, we obtain the monotonicity of the $\lambda$-functional along the ricci flow preserving conical singularities. And this monotonicity property of the $\lambda$-functional implies the no steady or expanding breather result for the ricci flow preserving conical singularities.

The $W$-functional functional and $\mu$-entropy on compact manifolds with isolated coincal singularities will be discussed in a subsequent paper.
\end{section}


\begin{section}{Manifolds with Isolated Conical Singularities}

As mentioned in the introduction, roughly speaking, a compact Riemannian manifold with isolated conical singularities is a singular manifold $(M, g)$ whose singular set $S$ consists of finite many points and its regular part $(M\setminus S, g)$ is a smooth Riemannian manifold. Moreover, near the singularities, the metric is asymptotic to a (finite) metric cone $C_{(0, 1]}(N)$ where $N$ is a compact smooth Riemannian manifold with metric $h_0$. More precisely,

\noindent  \begin{defn}\label{ManifoldWithConicalSingularities}
We say $(M^{n},d,g,p_{1},\cdots,p_{k})$ is a compact Riemannian manifold with isolated conical singularities at $p_{1},\cdots,p_{k}$, if
\begin{itemize}
\item $(M,d)$ is a compact metric space,
\item $(M_{0},g|_{M_{0}})$ is an n-dimensional smooth Riemannian manifold, and the Riemannian metric $g$ induces the given metric $d$ on $M_{0}$, where $M_{0}=M\setminus\{p_{1},\cdots,p_{k}\}$,
\item for each singularity $p_{i}$, $1\leq i\leq k$, their exists a neighborhood $U_{p_{i}}\subset M$ of $p_{i}$ such that $U_{p_{i}}\cap \{p_{1},\cdots, p_{k}\}=\{p_{i}\}$, $(U_{p_{i}}\setminus\{p_{i}\},g|_{U_{p_{i}}\setminus\{p_{i}\}})$ is isometric to $((0,\varepsilon_{i})\times N_{i},dr^{2}+r^{2}h_{r})$ for some $\varepsilon_{i}>0$ and a compact smooth manifold $N_{i}$, where $r$ is a coordinate on $(0, \varepsilon_{i})$ and $h_{r}$ is a smooth family of Riemannian metrics on $N_{i}$ satisfying $h_{r}=h_{0}+o(r^{\alpha_{i}})$ as $r\rightarrow 0$, where $\alpha_{i}>0$ and $h_{0}$ is a smooth Riemannian metric on $N_{i}$.
\end{itemize}
Moreover, we say a singularity $p$ is a cone-like singularity, if the metric $g$ on a neighborhood of $p$ is isometric to $dr^{2}+r^{2}h_{0}$ for some fixed metric $h_{0}$ on the cross section $N$.
\end{defn}

\begin{remark}
In the rest of this paper, we will only work on manifolds with a single conical point as there is no essential difference between the case of a single singular point  and that of multiple isolated singularities. And all our work and results for manifolds with a single conical point go through for manifolds with isolated conical singularities.
\end{remark}

The systematic study of manifolds with conical singularities began with the seminal work of Cheeger \cite{Che} and continued through a series of work of Cheeger and other people.  Most of the work concerns the Hodge-Laplacian and the analysis is done on the smooth part, using $L^2$ spaces. Since the Hodge-Laplacian is nonnegative, it has a natural self-adjoint extension, the Friedrichs extension. They are usually not the only self-adjoint extensions however, as the so called ideal boundary conditions of Cheeger give more subtle ones.

In our case, as usual, one does analysis away from the singular set. However, the existence of Friedrichs extension is much more subtle, and to control the behavior near the singularity, we employ weighted Sobolev spaces, defined in the next section.

\end{section}


\begin{section}{Weighted Sobolev Spaces}

\noindent In this section, we define certain weighted Sobolev spaces on compact Riemannian manifolds with conical singularities and establish the compact embedding property for the weighted Sobolev spaces.

Sobolev spaces are one of the basic tools to study elliptic operators on compact manifolds. However, usual Sobolev spaces and relevant results cannot be directly applied to elliptic operators on non-compact manifolds. Instead, weighted Sobolev spaces on complete non-compact manifolds, especially on $\R^{n}$ and asymptotically flat manifolds, have been extensively studied and used to study the Laplacian and other elliptic operators on non-compact manifolds, see, e.g. \cite{Bar86}, \cite{Can81}, \cite{CBC81}, \cite{CSCB78}, \cite{Loc81}, \cite{LP87} \cite{LM85}, \cite{McO79}, and \cite{NW73}. Weighted Sobolev spaces on various interesting domains in $\R^{n}$ have been intensively studied as well and used to study partial differential equations on these domains, see, e.g. \cite{KMR97}, \cite{Kuf85}, \cite{Tri78}, and \cite{Tur00}.

On $\R^{n}$ and asymptotically flat manifolds, the distance function from a fixed point gives a weight function in weighted norms. On a cone, the radial coordinate (the distance function from the tip) naturally provides a weight function. By using this weight function near the tip, B. Botvinnik and S. Preston introduced certain weighted Sobolev spaces on manifolds with tame conical singularities to study the conformal Laplacian on these singular manifolds in \cite{BP03}. Inspired by their work, we define similar weighted Sobolev spaces on manifolds with isolated conical singularities. Here we adopt weight indices as in \cite{Bar86} so that scaling technique can be conveniently used in $\S 8$.

Let $(M^{n},g,p)$ be a compact Riemannian manifold with a single conical singularity at $p$, and $U_{p}$ be a conical neighborhood of $p$ such that $(U_{p}\setminus \{p\},g|_{U_{p}\setminus \{p\}})$ is isometric to $((0,\epsilon)\times N, dr^{2}+r^{2}h_{r})$. For each $k\in\mathbb{N}$ and $\delta\in\mathbb{R}$, we define the weighted Sobolev space $H^{k}_{\delta}(M)$ to be the completion of $C^{\infty}_{0}(M\setminus\{p\})$ with respect to the weighted Sobolev norm
\begin{equation}\label{WSN}
\|u\|^{2}_{H^{k}_{\delta}(M)}=\int_{M}(\sum^{k}_{i=0}\chi^{2(\delta-i)+n}|\nabla^{i}u|^{2})d\vol_{g},
\end{equation}
where $\nabla^{i}u$ denotes the $i$-times covariant derivative of the function $u$, and $\chi\in C^{\infty}(M\setminus\{p\})$ is a positive weight function satisfying
\begin{equation}\label{WeightFunction}
\chi(q)=
\begin{cases} 1 & \text{if} \ \ q\in M\setminus U_{p}, \\ \frac{1}{r} & \text{if} \ \  r=dist(q,p)<\frac{\epsilon}{4},
\end{cases}
\end{equation}
and $0<(\chi(q))^{-1}\leq 1$ for all $q\in M\setminus\{p\}$.

For the simplicity of notations, we set $H^{k}(M)\equiv H^{k}_{k-\frac{n}{2}}(M)$. Then 
\begin{equation}\label{WSN1}
\|u\|^{2}_{H^{k}(M)}=\int_{M}(\sum^{k}_{i=0}\chi^{2(k-i)}|\nabla^{i}u|^{2})d\vol_{g}.
\end{equation}

We have the following compact embedding of $H^{k}(M)$ into $L^{2}(M)$.
\begin{thm}\label{compactness on manifolds}
The continuous embedding
\begin{equation}\label{CompactEmbedding}
i:H^{k}(M)\hookrightarrow L^{2}(M)
\end{equation}
is compact for each $k\in\mathbb{N}$.
\end{thm}

Clearly the key here is to prove the analogous compact embedding result on the model finite cones. Let $(C_{\epsilon}(N),g)=((0,\epsilon)\times N, dr^{2}+r^{2}h)$ be a finite cone, where $h$ is a Riemannian metric on compact manifold $N$. The weighted Sobolev  norm then becomes 
\begin{equation}\label{WSNFC}
\|u\|^{2}_{H^{k}(C_{\epsilon}(N))}=\int_{C_{\epsilon}(N)}(\sum^{k}_{i=0}\frac{1}{r^{2(k-i)}}|\nabla^{i}u|^{2})d\vol_{g}.
\end{equation}

We first establish the following
\begin{lem}\label{compactness on cones}
The continuous embedding
$$i:H^{k}(C_{\epsilon}(N))\hookrightarrow L^{2}(C_{\epsilon}(N))$$
is compact for each $k\in\mathbb{N}$.
\end{lem}
\begin{proof}
Clearly $\|u\|_{H^{k}(C_{\epsilon}(N))}\geq\|u\|_{H^{l}(C_{\epsilon}(N))}$, for $k\geq l\in\mathbb{N}$ (and $\epsilon \leq 1$; otherwise there is also a positive constant depending on $\epsilon$). Therefore, it suffices to show that the embedding:
$i:H^{1}(C_{\epsilon}(N))\hookrightarrow L^{2}(C_{\epsilon}(N))$, is compact. Let $(\widetilde{C}_{\epsilon}(N),\widetilde{g})=((0,\epsilon)\times N, dr^{2}+h)$ be a finite cylinder, and $W^{1,2}_{0}(\widetilde{C}_{\epsilon}(N))$ be the usual Sobolev space on the cylinder $\tilde{C}_{\epsilon}(N)$, which is the completion of $C^{\infty}_{0}(\widetilde{C}_{\epsilon}(N))$ with respect to the norm:
$$\|u\|_{W^{1,2}(\widetilde{C}_{\epsilon}(N))}=\int_{\widetilde{C}_{\epsilon}(N)}(u^{2}+|\widetilde{\nabla}u|^{2}_{\widetilde{g}})d\vol_{\widetilde{g}},$$
where $\widetilde{\nabla}u$ is the gradient of $u$ with respect to the metric $\widetilde{g}$. It is clear that the mapping:
\begin{eqnarray*}
L^{2}(C_{\epsilon}(N),g) & \rightarrow & L^{2}(\widetilde{C}_{\epsilon}(N),\widetilde{g}) \\
u & \mapsto & \widetilde{u}=r^{\frac{n-1}{2}}u
\end{eqnarray*}
is unitary, where $n=dim(N)+1$. We will show that
\begin{equation}\label{cone cylinder inequality}
\|u\|_{H^{1}(C_{\epsilon}(N))}\geq\frac{3}{4}\min\{1, \frac{1}{\epsilon^{2}}\}\|\widetilde{u}\|_{W^{1,2}(\widetilde{C}_{\epsilon}(N))},
\end{equation}
for all $u\in C^{\infty}_{0}((0,\epsilon)\times N)$. This then completes the proof, since the embedding
$$W^{1,2}_{0}(\widetilde{C}_{\epsilon}(N))\hookrightarrow L^{2}(\widetilde{C}_{\epsilon}(N),\widetilde{g})$$
is compact by the classical Rellich Lemma.

Now we prove the inequality $(\ref{cone cylinder inequality})$. Let $0=\mu_{1}<\mu_{2}\leq\mu_{3}\leq\cdots\nearrow+\infty$ be the eigenvalues (counted with multiplicity) of the Laplacian, $-\Delta_{N}$, on the compact Riemannian manifold $(N,h)$, and $\psi_{1},\psi_{2},\psi_{3},\cdots$ be corresponding orthonormal eigenfunctions. Let $u\in C^{\infty}_{0}((0,\epsilon)\times N)$. We expand the function $u$ and $\widetilde{u}$, respectively, as
\begin{eqnarray}
u(r,x) &=& \sum^{\infty}_{i=1}u_{i}(r)\psi_{i}(x), \cr
\widetilde{u}(r,x) &=& \sum^{\infty}_{i=1}\widetilde{u}_{i}(r)\psi_{i}(x),
\end{eqnarray}
where $u_{i}(r)=r^{-\frac{n-1}{2}}\widetilde{u}_{i}(r)$.

\begin{align*}
\|u\|^{2}_{H^{1}(C_{\epsilon}(N))} &=\int_{C_{\epsilon}(N)}(\frac{1}{r^{2}}u^{2}+|\nabla u|^{2}_{g})dvol_{g}=\int_{C_{\epsilon}(N)}(\frac{1}{r^{2}}u^{2}+|\partial_{r}u|^{2}+\frac{1}{r^{2}}|\nabla_{N}u|^{2}_{h})d\vol_{g}\\
& =\int^{\epsilon}_{0}\int_{N}[\frac{1}{r^{2}}(\sum^{\infty}_{i=1}u_{i}(r)\psi_{i}(x))^{2}
  +(\sum^{\infty}_{i=1}u^{'}_{i}(r)\psi_{i}(x))^{2}\\
&\ \ \ +\frac{1}{r^{2}}(\sum^{\infty}_{i=1}u_{i}(r)\nabla_{N}\psi_{i}(x))^{2}]r^{n-1}d\vol_{h}dr\\
& =\int^{\epsilon}_{0}[\frac{1}{r^{2}}\sum^{\infty}_{i=1}(u_{i}(r))^{2}
  +\sum^{\infty}_{i=1}(u^{'}_{i}(r))^{2}
  +\frac{1}{r^{2}}\sum^{\infty}_{i=1}\mu_{i}(u_{i}(r))^{2}]r^{n-1}dr\\
& =\int^{\epsilon}_{0}[\frac{1}{r^{2}}\sum^{\infty}_{i=1}(\widetilde{u}_{i}(r))^{2}
  +\sum^{\infty}_{i=1}(-\frac{n-1}{2}\frac{1}{r}\widetilde{u}_i(r)+\widetilde{u}^{'}_{i}(r))^{2}
  +\frac{1}{r^{2}}\sum^{\infty}_{i=1}\mu_{i}(\widetilde{u}_{i}(r))^{2}]dr\\
& =\int^{\epsilon}_{0}[\frac{1}{r^{2}}\sum^{\infty}_{i=1}(1+\frac{(n-1)(n-3)}{4}
  +\mu_{i})(\widetilde{u}_{i}(r))^{2}+\sum^{\infty}_{i=1}(\widetilde{u}^{'}_{i})^{2}]dr\\
& \geq\int^{\epsilon}_{0}[\sum^{\infty}_{i=1}(
\frac{3}{4}+\mu_{i})\frac{1}{r^{2}}(\widetilde{u}_{i}(r))^{2}+\sum^{\infty}_{i=1}(\widetilde{u}^{'}_{i})^{2}]dr \geq\frac{3}{4}\min\{1, \frac{1}{\epsilon^{2}}\}\|\widetilde{u}\|^{2}_{W^{1,2}(\widetilde{C}(N))}
\end{align*}

Here in the above we have used the identity $\int^{\epsilon}_{0}\frac{1}{r}\widetilde{u}_i(r)\widetilde{u}^{'}_{i}(r)\, dr = 1/2 \int^{\epsilon}_{0} \frac{1}{r^2}
(\widetilde{u}_i(r))^2 \, dr$.
\end{proof}

{\it{Proof of Theorem $\ref{compactness on manifolds}$.}} As in the proof of Lemma $\ref{compactness on cones}$, it suffices to show that $H^{1}(M)\hookrightarrow L^{2}(M)$ is compact. Because $(U_{p}\setminus\{p\}, g|_{U_{p}\setminus\{p\}})$ is isometric to $((0,\epsilon)\times N, dr^{2}+r^{2}h_{r})$, where $h_{r}=h_{0}+o(r^{\alpha})$, for some $\alpha>0$, if we define $g_{0}=dr^{2}+r^{2}h_{0}$ on $(0,\epsilon)\times N$, there exists $0<\epsilon_{1}<\frac{\epsilon}{4}$, such that on $(0,\epsilon_{1})\times N$,
$$\frac{1}{2}g_{0}\leq g \leq 2g_{0}.$$
Then for any $u\in C^{\infty}_{0}((0,\epsilon)\times N)$, we have
\begin{equation}\label{inequality for H-space}
\frac{1}{2^{1+\frac{n}{2}}}\|u\|^{2}_{H^{1}(C_{\epsilon_{1}}(N), g_{0})}\leq\|u\|^{2}_{H^{1}(C_{\epsilon_{1}}(N), g)}
\leq2^{1+\frac{n}{2}}\|u\|^{2}_{H^{1}(C_{\epsilon_{1}}(N), g_{0})},
\end{equation}
\begin{equation}\label{inequality for L-space}
\frac{1}{2^{\frac{n}{2}}}\|u\|^{2}_{L^{2}(C_{\epsilon_{1}}(N), g_{0})}\leq\|u\|^{2}_{L^{2}(C_{\epsilon_{1}}(N), g)}
\leq2^{\frac{n}{2}}\|u\|^{2}_{L^{2}(C_{\epsilon_{1}}(N), g_{0})}.
\end{equation}
By Lemma $\ref{compactness on cones}$, inequalities $(\ref{inequality for H-space})$ and $(\ref{inequality for L-space})$ imply that the embedding \begin{equation}\label{compactness on conical part}
H^{1}(C_{\epsilon_{1}}(N),g)\hookrightarrow L^{2}(C_{\epsilon_{1}}(N),g)
\end{equation}
is compact. Set $M_{0}=M\setminus(0,\frac{\epsilon_{1}}{2})\times N$. The compactness of embedding $W^{1,2}_{0}(M_{0})\hookrightarrow L^{2}(M_{0})$ and the compactness of the embedding $(\ref{compactness on conical part})$ imply the compactness of the embedding $H^{1}(M)\hookrightarrow L^{2}(M)$.

\end{section}


\begin{section}{Spectrum of $-4\Delta+R$ on a finite cone}

\noindent In this section, we study the spectrum of the Schr\"{o}dinger operator $L=-4\Delta+R$ on a finite cone $(C_{\epsilon}(N),g)=((0,\epsilon)\times N,dr^{2}+r^{2}h_{0})$ with the Dirichlet boundary condition (on the smooth boundary). By using the compact embedding results obtained in the previous section and establishing a semi-boundedness estimate for the operator $L$, we show that the spectrum of the Friedrichs extension of $L$ on a finite cone with the Dirichlet boundary condition consists of discrete eigenvalues with finite multiplicities.

Let
$$L=-4\Delta_{g}+R_{g}: L^{2}(C_{\epsilon}(N))\rightarrow L^{2}(C_{\epsilon}(N))$$
be a densely defined unbounded operator with the domain $Dom(L)=C^{\infty}_{0}(C_{\epsilon}(N))$.

\begin{thm}\label{Semi-boundednessOnCones}
If the scalar curvature $R_{h_{0}}$ on the cross section $(N^{n-1},h_{0})$ satisfies $R_{h_{0}}>(n-2)$, then
$$(Lu,u)_{L^{2}}\geq\delta_{0}\|u\|_{H^{1}(C_{\epsilon}(N))}$$
for all $u\in C^{\infty}_{0}(C_{\epsilon}(N))$, and some constant $\delta_{0}>0$ depending on $\min\limits_{x\in N}\{R_{h_{0}}(x)\}$ and $n$. In particular, the operator $(L, Dom(L)=C^{\infty}_{0}(C_{\epsilon}(N)))$ is strictly positive.
\end{thm}
\begin{proof}
Because the manifold $(N^{n-1},h)$ is compact, and $R_{h_{0}}>(n-2)$, we have
$$\min\limits_{x\in N}\{R_{h_{0}}(x)\}>(n-2).$$
And because
$$(n-1)(n-2)-\frac{4-\delta}{4}[(n-1)(n-3)+1]+\delta\rightarrow n-2, \ \ \ \text{as} \ \ \delta\searrow0,$$
there exists $\delta_{0}>0$, such that
\begin{equation}\label{lower bound of scalar curvature on cross section}
\min_{x\in N}\{R_{h_{0}}(x)\}>(n-1)(n-2)-\frac{4-\delta_{0}}{4}[(n-1)(n-3)+1]+\delta_{0}.
\end{equation}

Set $$L_{\delta_{0}}=-(4-\delta_{0})\Delta_{g}+R_{g}-\frac{1}{r^{2}}\delta_{0}.$$
Then $$L=L_{\delta_{0}}-\delta_{0}\Delta_{g}+\frac{1}{r^{2}}\delta_{0},$$
and for any $u\in C^{\infty}_{0}(C_{\epsilon}(N))$,
\begin{align*}
(Lu,u)_{L^{2}}
&=\int_{C_{\epsilon}(N)}(Lu)u\, d\vol_{g}\\
&=\int_{C_{\epsilon}(N)}(L_{\delta_{0}}u)u\, d\vol_{g}\\
&\ \ \ +\int_{C_{\epsilon}(N)}[(-\delta_{0}\Delta_{g} u)u+\frac{1}{r^{2}}\delta_{0}u^{2}]d\vol_{g}\\
&=\int_{C_{\epsilon}(N)}(L_{\delta_{0}}u)u\, d\vol_{g}\\
&\ \ \ +\delta_{0}\int_{C_{\epsilon}(N)}(|\nabla u|^{2}+\frac{1}{r^{2}}u^{2})d\vol_{g}\\
&=(L_{\delta_{0}}u,u)_{L^{2}}+\delta_{0}\|u\|_{H^{1}(C_{\epsilon}(N))}.
\end{align*}
Thus it suffices to show that $(L_{\delta_{0}}u,u)_{L^{2}}\geq0$.

Indeed, we claim that
\begin{equation}\label{postivity of L_delta}
(L_{\delta_{0}}u,u)_{L^{2}}\geq C\|u\|_{L^{2}},
\end{equation}
for all $u\in C^{\infty}_{0}(C_{\epsilon}(N))$,
where
$$C=\min\{\min\limits_{x\in N}\{R_{h_{0}}(x)\}-[(n-1)(n-2)-\frac{4-\delta_{0}}{4}((n-1)(n-3)+1)+\delta_{0}],1\}>0.$$

Now we prove the claim $(\ref{postivity of L_delta})$. For any $u\in C^{\infty}_{0}(C_{\epsilon}(N))$, we similarly expand it  in terms of an orthonormal basis of eigenfunctions $\psi_{i}(x)$, this time,  of operator $-(4-\delta_{0})\Delta_{h_{0}}+R_{h_{0}}-\delta_{0}$ on $N^{n-1}$,  with eigenvalues $\mu_{i}$,
\begin{equation}\label{expansion of smooth functions}
u=\sum^{\infty}_{i=0}u_{i}(r)\varphi_{i}(x).
\end{equation}
Then by using $(\ref{Lapace_on_cone})$ and $(\ref{Scale_curvature_on_cone})$,

\begin{equation*}
L_{\delta_{0}}u=\sum^{\infty}_{i=0}\{-(4-\delta)u''_{i}(r)-(4-\delta)\frac{n-1}{r}u'_{i}(r)-\frac{1}{r^{2}}[-\mu_{i}+(n-1)(n-2)]u_{i}(r)\}\psi_{i}.
\end{equation*}

Again $\tilde{u}_{i}(r)=r^{\frac{n-1}{2}}u_{i}(r)$, and hence
\begin{equation*}\label{expansion for L}
L_{\delta_{0}}u=\sum^{\infty}_{i=0}\{-(4-\delta_{0})\tilde{u}''_{i}+\frac{1}{r^{2}}[\mu_{i}-(n-1)(n-2)+\frac{4-\delta_{0}}{4}(n-1)(n-3)]\tilde{u}_{i}(r)\}
r^{-\frac{n-1}{2}}\psi_{i}.
\end{equation*}

Because $\mu_{i}\rightarrow +\infty$ as $i\rightarrow +\infty$, we can take large enough $i_{0}\in \mathbb{N}$ such that for all $i\geq i_{0}, \ \  \mu_{i}-(n-1)(n-2)+\frac{4-\delta_{0}}{4}(n-1)(n-3)>1$.

\begin{align*}
(L_{\delta_{0}}u,u)_{L^{2}} &= \int^{\epsilon}_{0}\int_{N}\{\sum^{\infty}_{i=0}[-(4-\delta_{0})\tilde{u}''_{i}(r)+\frac{1}{r^{2}}(\mu_{i}-(n-1)(n-2)\\
&\ \ \ \ +\frac{4-\delta_{0}}{4}(n-1)(n-3))\tilde{u}_{i}(r)]r^{-\frac{n-1}{2}}\psi_{i}\}
\{\sum^{\infty}_{j=0}\tilde{u}_{j}(r)r^{-\frac{n-1}{2}}\psi_{i}\}r^{n-1}d\vol_{h_{0}}dr\\
&= \int^{\epsilon}_{0}\sum^{\infty}_{i=0}(4-\delta_{0})(\tilde{u}'_{i}(r))^{2}dr\\
&\ \ \ +\int^{\epsilon}_{0}\sum^{\infty}_{i=0}\{\frac{1}{r^{2}}[\mu_{i}-(n-1)(n-2)+\frac{4-\delta_{0}}{4}(n-1)(n-3)]
(\tilde{u}_{i}(r))^{2}\}dr\\
&= \int^{\epsilon}_{0}\sum^{i_{0}}_{i=0}(4-\delta_{0})(\tilde{u}'_{i}(r))^{2}dr\\
&\ \ \  +\int^{\epsilon}_{0}\sum^{i_{0}}_{i=0}\{\frac{1}{r^{2}}[\mu_{i}-(n-1)(n-2)+\frac{4-\delta_{0}}{4}(n-1)(n-3)](\tilde{u}_{i}(r))^{2}\}dr\\
&\ \ \  + \int^{\epsilon}_{0}\sum^{\infty}_{i=i_{0}+1}(4-\delta_{0})(\tilde{u}'_{i}(r))^{2}dr\\
&\ \ \  +\int^{\epsilon}_{0}\sum^{\infty}_{i=i_{0}+1}\{\frac{1}{r^{2}}[\mu_{i}-(n-1)(n-2)+\frac{4-\delta_{0}}{4}(n-1)(n-3)](\tilde{u}_{i}(r))^{2}\}dr\\
\end{align*}

Denote by $ \uppercase\expandafter{\romannumeral1}$ the sum of the first two terms and $ \uppercase\expandafter{\romannumeral2}$ that of the last two terms.

By using the Hardy's inequality,
\begin{align*}
\uppercase\expandafter{\romannumeral1} &= \int^{\epsilon}_{0}\sum^{i_{0}}_{i=0}(4-\delta_{0})(\tilde{u}'_{i}(r))^{2}dr\\
&\ \ \  +\int^{\epsilon}_{0}\sum^{i_{0}}_{i=0}\{\frac{1}{r^{2}}[\mu_{i}-(n-1)(n-2)+\frac{4-\delta_{0}}{4}(n-1)(n-3)](\tilde{u}_{i}(r))^{2}\}dr\\
&\geq \int^{\epsilon}_{0}\sum^{i_{0}}_{i=0}\frac{4-\delta_{0}}{4}\frac{1}{r^{2}}(\tilde{u}_{i}(r))^{2}dr\\
&\ \ \  +\int^{\epsilon}_{0}\sum^{i_{0}}_{i=0}\{\frac{1}{r^{2}}[\mu_{i}-(n-1)(n-2)+\frac{4-\delta_{0}}{4}(n-1)(n-3)](\tilde{u}_{i}(r))^{2}\}dr\\
&\geq \{\min\limits_{x\in N}\{R_{h_{0}}(x)\}-[(n-1)(n-2)\\
&\ \ \ -\frac{4-\delta_{0}}{4}((n-1)(n-3)+1)+\delta_{0}]\}
\int^{\epsilon}_{0}\frac{1}{r^{2}}\sum^{i_{0}}_{i=0}(\tilde{u}_{i}(r))^{2}dr\\
&\geq C\int^{\epsilon}_{0}\sum^{i_{0}}_{i=0}(\tilde{u}_{i}(r))^{2}dr.
\end{align*}

On the other hand,  since $\mu_{i}-(n-1)(n-2)+\frac{4-\delta_{0}}{4}(n-1)(n-3)>1$ for all $i>i_{0}$,
\begin{align*}
\uppercase\expandafter{\romannumeral2} &=\int^{\epsilon}_{0}\sum^{\infty}_{i=i_{0}+1}(4-\delta_{0})(\tilde{u}'_{i}(r))^{2}dr\\
&\ \ \  +\int^{\epsilon}_{0}\sum^{\infty}_{i=i_{0}+1}\{\frac{1}{r^{2}}[\mu_{i}-(n-1)(n-2)+\frac{4-\delta_{0}}{4}(n-1)(n-3)](\tilde{u}_{i}(r))^{2}\}dr\\
&\geq \int^{\epsilon}_{0}\sum^{\infty}_{i=i_{0}}(\tilde{u}_{i}(r))^{2}dr\\
&\geq C\int^{\epsilon}_{0}\sum^{\infty}_{i=i_{0}}(\tilde{u}_{i}(r))^{2}dr.
\end{align*}
This proves the claim $(\ref{postivity of L_delta})$. And the proof is complete.
\end{proof}

\begin{cor}\label{Friedrichs extension on cones}
If the scalar curvature $R_{h_{0}}$ on $(N^{n-1},h_{0})$ satisfies $R_{h_{0}}>(n-2)$, then the operator $(L, Dom(L)=C^{\infty}_{0}(C_{\epsilon}(N)))$ has a self-adjoint strictly positive Friedrichs extension $(\widetilde{L}, Dom(\widetilde{L}))$. Moreover, $Dom(\widetilde{L})\subset H^{1}(C_{\epsilon}(N))$, $\widetilde{L}$ is injective, and the image $Ran(\widetilde{L})=L^{2}(C_{\epsilon}(N))$.
\end{cor}
\begin{proof}
The existence of the self-adjoint strictly positive and surjective extension follows from the Neumann Theorem in \cite{EK}, because the operator $(L, Dom(L))$ is strictly positive by Theorem $\ref{Semi-boundednessOnCones}$. Moreover, from Theorem $\ref{Semi-boundednessOnCones}$, we can obtain that the completion of $C^{\infty}_{0}(C_{\epsilon}(N))$ with respect to the norm $\|u\|_{L}=(Lu,u)_{L^{2}}$ is a subspace of $H^{1}(C_{\epsilon}(N))$. Thus from the construction of the Friedrichs extension in the proof of the Neumann theorem in \cite{EK}, we can easily see that $Dom(\widetilde{L})\subset H^{1}(C_{\epsilon}(N))$.
\end{proof}

\begin{thm}\label{spectrum on cones}
If the scalar curvature of $(N^{n-1},h_{0})$, $R_{h_{0}}>(n-2)$, then the spectrum of the Friedrichs extension of the operator $-4\Delta_{g}+R_{g}$ on $(C_{\epsilon}(N),g=dr^{2}+r^{2}h_{0})$ consists of discrete eigenvalues with finite multiplicities
$$\lambda_{1}\leq\lambda_{2}\leq\lambda_{3}\leq\cdots,$$
and $\lambda_{k}\rightarrow +\infty$ as $k\rightarrow+\infty$. Moreover,  the corresponding eigenfunctions $\{\varphi_{i}\}^{\infty}_{i=1}$ form a complete basis of $L^{2}(C_{\epsilon}(N))$.
\end{thm}
\begin{proof}
By the Corollary $\ref{Friedrichs extension on cones}$, the Friedrichs extension $\widetilde{L}: Dom(\widetilde{L})\rightarrow L^{2}(C_{\epsilon}(N))$ is one-to-one and onto. And its inverse
$$\widetilde{L}^{-1}: L^{2}(C_{\epsilon}(N))\rightarrow Dom(\widetilde{L})\hookrightarrow H^{1}(C_{\epsilon}(N))\hookrightarrow L^{2}(C_{\epsilon}(N))$$
is a self-adjoint compact operator, because the embeddding $H^{1}(C_{\epsilon}(N))\hookrightarrow L^{2}(C_{\epsilon}(N))$ is compact. Then the spectrum theorem of self-adjoint compact operators completes the proof.
\end{proof}
\end{section}


\begin{section}{Spectrum of $-4\Delta+R$ on compact manifolds with a single conical singularity}

\noindent In this section, we study the spectrum of the Schr\"{o}dinger operator $-4\Delta+R$ on compact Riemannian manifolds with a single conical singularity. By using the semi-boundedness estimate for the operator $-4\Delta+R$ on finite cones in Theorem $\ref{Semi-boundednessOnCones}$, we establish a similar estimate for the operator $-4\Delta+R$ on compact Riemannian manifolds with a single conical singularity. Then, we prove that the spectrum of the Friderichs extension of the operator $-4\Delta+R$ on compact Riemannian manifolds with a single conical singularity consists of discrete eigenvalues with finite multiplicities.

\begin{thm}\label{Semi-boundedness}
Let $(M^{n},g,p)$ be a compact Riemannian manifold with a single conical singularity at $p$. If the scalar curvature $R_{h_{0}}$ on $(N^{n-1},h_{0})$ satisfies $R_{h_{0}}>(n-2)$, then there exists a large enough constant $A$, such that the operator $L_{A}=-4\Delta_{g}+R_{g}+A$ satisfies:
$$(L_{A}u,u)_{L^{2}(M)}\geq C\|u\|_{H^{1}(M)}$$
for all $u\in C^{\infty}_{0}(M\setminus\{p\})$ and some constant $C>0$. In particular, the operator $(L_{A}, Dom(L_{A})=C^{\infty}_{0}(M\setminus\{p\}))$ is strictly positive.
\end{thm}
\begin{proof}
The conical neighborhood $(U_{p}\setminus\{p\},g|_{U_{p}\setminus\{p\}})$ of conical singularity $p$ is isometric to $((0,\epsilon)\times N, dr^{2}+r^{2}h_{r})$, where $h_{r}=h_{0}+o(r^{\alpha})$, for some $\alpha>0$. Then the scalar curvature on the conical neighborhood is given by
\begin{equation}\label{scalar curvature of asymptotic cone}
\begin{aligned}
R_{g}
&=\frac{1}{r^{2}}(R_{h_{r}}-(n-1)(n-2)+o(r^{\alpha}))\\
&=\frac{1}{r^{2}}(R_{h_{0}}-(n-1)(n-2)+o(r^{\alpha})).
\end{aligned}
\end{equation}
Because $R_{h_{0}}>(n-2)$, there exists $\beta(n)\in(0,1)$ such that
$$R_{h_{0}}>(n-1)[\frac{1}{\beta(n)^{2n+2}}(n-2)-(n-3)]-1.$$
Then there exists $\epsilon(n)>0$, such that on $(0,\epsilon(n))\times N$,
$$\beta(n)^{2}g_{0}\leq g\leq \frac{1}{\beta(n)^{2}}g_{0},$$
$$\beta(n)R_{h_{0}}\leq r^{2}R_{g}+(n-1)(n-2)\leq\frac{1}{\beta(n)}R_{h_{0}}.$$
For any $u\in C^{\infty}_{0}((0,\epsilon(n))\times N)$, we have
\begin{align*}
(Lu,u)_{L^{2}(C_{\epsilon(n)}(N))}
&=\int_{C_{\epsilon(n)}(N)}(-4\Delta u+Ru)u\, d\vol_{g}\\
&=\int_{C_{\epsilon(n)}(N)}(4|\nabla u|^{2}+Ru^{2})d\vol_{g}\\
&\geq \int_{C_{\epsilon(n)}(N)}[4\beta(n)^{n+1}|\nabla u|^{2}_{g_{0}}+\frac{1}{r^{2}}\beta(n)^{n+1}R_{h_{0}}u^{2}\\
&\ \ \ +\frac{1}{r^{2}}(-(n-1)(n-2))\frac{1}{\beta(n)^{n+1}}]d\vol_{g_{0}}\\
&=\beta(n)^{n+1}\int_{C_{\epsilon(n)}(N)}[-4\Delta_{g_{0}}u\\
&\ \ \ +\frac{1}{r^{2}}(R_{g_{0}}-\beta(n)^{-2n-2}(n-1)(n-2))u]u\, d\vol_{g_{0}}\\
&\geq \beta(n)^{n+1}C_{1}\|u\|_{H^{1}(C_{\epsilon(n)}(N))}
\end{align*}
The last inequality follows the same argument as in Theorem $\ref{Semi-boundednessOnCones}$, i.e. for any $u\in C^{\infty}_{0}((0,\epsilon(n))\times N)$,
\begin{equation}\label{postivity on asymptotic cone}
(Lu,u)_{L^{2}(C_{\epsilon(n)}(N))}\geq \beta(n)^{n+1}C_{1}\|u\|_{H^{1}(C_{\epsilon(n)}(N))}
\end{equation}

We cover the manifold $M$ by the conical neighborhood $(0,\epsilon(n))\times N$ of singularity $p$ and the interior part $M_{0}=M\backslash
C_{(0,\frac{1}{8}\epsilon(n))}(N)$. We construct a partition of unity subordinate to this covering as following. Let $\rho_{1}$ a function on $C_{\epsilon}(N)$ satisfying
\begin{equation*}
\rho_{1}(r,x)=
\begin{cases}
1, & 0<r<\frac{\epsilon(n)}{4}\\ 0, & r>\frac{\epsilon(n)}{2},
\end{cases}
\end{equation*}
with $0\leq\rho_{1}(r,x)\leq1$. We extend $\rho_{1}$ trivially to the whole $M$, i.e. set $\rho_{1}|_{M\setminus C_{\epsilon(n)}(N)}\equiv0$, and we still use $\rho_{1}$ to denote the extended function. Let $\rho_{2}=1-\rho_{1}$. Then $\{\rho_{1},\rho_{2}\}$ is a partition of unity subordinate to the covering $\{C_{\epsilon(n)(N)}, M_{0}\}$.

For any $u\in C^{\infty}_{0}(M)$,
\begin{align*}
(L_{B}u,u)
&=\int_{M}(L_{B}u_{1}+L_{B}u_{2})(u_{1}+u_{2})d\vol_{g}\\
&=\int_{M}(L_{B}u_{1})u_{1}dvol_{g}+\int_{M}(L_{B}u_{1})u_{2}d\vol_{g}\\
&\ \ \ +\int_{M}(L_{B}u_{2})u_{1}dvol_{g}+\int_{M}(L_{B}u_{2})u_{2}d\vol_{g},
\end{align*}
where $u_{1}=\rho_{1}u$, $u_{2}=\rho_{2}u$, and $L_{B}=L+B$ for some $B>0$.

By $(\ref{postivity on asymptotic cone})$, we have
\begin{align*}
\int_{M}(L_{B}u_{1})u_{1}d\vol_{g} & \geq \beta(n)^{n}C_{1}\int_{M}(\chi^{2}|u_{1}|^{2}+|\nabla u_{1}|^{2})d\vol_{g},
\end{align*}
where $C_{1}$ is a positive constant.

Because $u_{2}$ is compactly supported in $M_{0}$ and $R$ is bounded on $\overline{M_{0}}$, i.e.  there exists $C_{2}<0$ such that $R>C_{2}$ on $M_{0}$, we have
\begin{align*}
\int_{M}(L_{B}u_{2})u_{2}d\vol_{g} & =\int_{M_{0}}(-4\Delta u_{2}+(R+B)u_{2})u_{2}d\vol_{g}\\
                              & =\int_{M_{0}}(4|\nabla u_{2}|^{2}+(R+B)|u_{2}|^{2})d\vol_{g}\\
                              & \geq C_{2}\int_{M_{0}}(|\nabla u_{2}|^{2}+\chi^{2}|u_{2}|^{2})d\vol_{g}\\
\end{align*}

By integration by parts,
\begin{align*}
\int_{M}(L_{B}u_{1})u_{2}d\vol_{g}
&=\int_{M}(L_{B}u_{2})u_{1}d\vol_{g}\\
& =\int_{M}\langle\nabla u_{1},\nabla u_{2}\rangle d\vol_{g}+\int_{M}(R+B)u_{1}u_{2}d\vol_{g}\\
& =\int_{M}\langle u\nabla\rho_{1}+\rho_{1}\nabla u,u\nabla\rho_{2}+\rho_{2}\nabla u\rangle d\vol_{g}\\
&\ \ \ +\int_{M}(R+B)u_{1}u_{2}d\vol_{g}\\
& =\int_{M}u^{2}(\partial_{r}\rho_{1})(\partial_{r}\rho_{2})d\vol_{g} +\int_{C_{\epsilon}(N)}u\rho_{2}(\partial_{r}\rho_{1})(\partial_{r}u)d\vol_{g}\\
& \ \ \  +\int_{C_{\epsilon}(N)}u\rho_{1}(\partial_{r}\rho_{2})(\partial_{r}u)d\vol_{g}+\int_{C_{\epsilon}(N)}\rho_{1}\rho_{2}|\nabla u|^{2}d\vol_{g}\\
&\ \ \ +\int_{M}(R+B)u_{1}u_{2}d\vol_{g}.
\end{align*}

Then we have
$$\int_{M}u^{2}(\partial_{r}\rho_{1})(\partial_{r}\rho_{2})dvol_{g}>C_{3}\int_{M}u^{2}d\vol_{g},$$
for some negative constant $C_{3}$, and 
\begin{align*}
\int_{C_{\epsilon}(N)}u\rho_{2}(\partial_{r}\rho_{1})(\partial_{r}u)d\vol_{g}
& =\int^{\epsilon}_{0}\int_{N}u\rho_{2}(\partial_{r}\rho_{1})(\partial_{r}u)r^{n-1}d\vol_{h_{r}}dr\\
& =-\frac{1}{2}\int^{\epsilon}_{0}\int_{N}(\partial_{r}\rho_{2})(\partial_{r}\rho_{1})u^{2}r^{n-1}d\vol_{h_{r}}dr\\
& \ \ \ -\frac{1}{2}\int^{\epsilon}_{0}\int_{N}\rho_{2}(\partial^{2}_{r}\rho_{1})u^{2}r^{n-1}d\vol_{h_{r}}dr\\
& \ \ \ -\frac{1}{2}\int^{\epsilon}_{0}\int_{N}u^{2}\frac{\rho_{2}(\partial_{r}\rho_{1})}{r}(n-1)r^{n-1}d\vol_{h_{r}}dr\\
& \ \ \ -\frac{1}{2}\int^{\epsilon}_{0}\int_{N}u^{2}\rho_{2}(\partial_{r}\rho_{1})tr(h^{-1}_{r}\frac{\partial}{\partial r}h_{r})r^{n-1}d\vol_{h_{r}}dr\\
& >C_{4}\int_{M}u^{2}d\vol_{g},\\
\end{align*}
for some negative constant $C_{4}$.

Similarly, we have
$$\int_{C_{\epsilon}(N)}u\rho_{1}(\partial_{r}\rho_{2})(\partial_{r}u)d\vol_{g}>C_{5}\int_{M}u^{2} d\vol_{g},$$
for some negative constant $C_{5}$.

Thus
\begin{align*}
\int_{M}(L_{B}u_{1})u_{2}d\vol_{g}
&>\int_{M}(\rho_{1}\rho_{2}|\nabla u|^{2}+(R+B)u_{1}u_{2})d\vol_{g}\\
&\ \ \ +(C_{3}+C_{4}+C_{5})\int_{M}u^{2}d\vol_{g}\\
&>\int_{M}(\rho_{1}\rho_{2}|\nabla u|^{2}+u_{1}u_{2})d\vol_{g}\\
&\ \ \ +C_{6}\int_{M}u^{2}d\vol_{g},
\end{align*}
where, $C_{6}=C_{3}+C_{4}+C_{5}$

In the same way one also derives 
\begin{align*}
(u_{1},u_{2})_{H^{1}_{2}(M)} &<C_{7}\int_{M}(\rho_{1}\rho_{2}|\nabla u|^{2}+u_{1}u_{2})d\vol+C_{8}\int_{M}u^{2}d\vol_{g},
\end{align*}
for some $C_{7}>0$, such that $\frac{1}{C_{7}}<\beta(n)^{n}C_{1},C_{2}$, and $C_{8}>0$

Thus
$$\int_{M}(L_{B}u_{1})u_{2}d\vol_{g}>\frac{1}{C_{7}}(u_{1},u_{2})_{H^{1}_{2}(M)}+(C_{6}-\frac{C_{8}}{C_{7}})\int_{M}u^{2}d\vol_{g},$$
and therefore,
$$\int_{M}(L_{B}u)ud\vol_{g}>\frac{1}{C_{7}}(u,u)_{H^{1}_{2}(M)}+2(C_{6}-\frac{C_{8}}{C_{7}})\int_{M}u^{2}d\vol_{g}.$$

Let $A=B+2(\frac{C_{8}}{C_{7}}-C_{6})$, then we have
$$\int_{M}(L_{A}u)u\, d\vol_{g}>\frac{1}{C_{7}}(u,u)_{H^{1}_{2}(M)}.$$

In particular,
$$(L_{A}u,u)_{L^{2}}>\frac{1}{C_{7}}\|u\|^{2}_{L^{2}},$$
i.e.  $(L_{A},Dom(L_{A})=C^{\infty}_{0}(M))$ is strictly positive.
\end{proof}

By Theorem $\ref{compactness on manifolds}$, Theorem $\ref{Semi-boundedness}$, and applying the similar argument as in the proof of Theorem $\ref{spectrum on cones}$ to the operator $L_{A}=-4\Delta_{g}+R_{g}+A$, we obtain the following spectrum property for $-4\Delta_{g}+R_{g}$, as the spectrum of $-4\Delta_{g}+R_{g}$ is a constant shift of the spectrum of $L_{A}$.
\begin{thm}\label{SpectrumProperty}
Let $(M,g,p)$ be a compact Riemannian manifold with a single conical singularity $p$. If the scalar curvature $R_{h_{0}}$ on $(N^{n-1},h_{0})$ satisfies $R_{h_{0}}>(n-2)$, then the spectrum of the Friedrichs extension of the operator $-4\Delta+R$ on $(M,g,p)$ consists of discrete eigenvalues with finite multiplicity
$$\lambda_{1}\leq\lambda_{2}\leq\lambda_{3}\leq\cdots,$$
and $\lambda_{k}\rightarrow+\infty$. And the corresponding eigenfunctions $\{\varphi_{i}\}^{\infty}_{i=1}$ form a complete basis of $L^{2}(M)$. Moreover, each eigenfunction $\varphi_{i} \in C^{\infty}(M\setminus\{p\})$ and satisfied the usual eigenvalue equation on $M\setminus\{p\}$.
\end{thm}

The last part is an easy consequence of the standard elliptic regularity theory for weak solutions.
\end{section}


\begin{section}{Min-max principle for eigenvalues of $-4\Delta+R+A$}

The first eigenvalue of the Laplacian on compact manifolds (with certain boundary condition if the boundary is not empty) can be characterized as the infimum of the Rayleigh quotient over the space of admissible functions. And other eigenvalues can be characterized similarly. This is usually known as the min-max principle of eigenvalues of the Laplacian. The min-max principle is a fundamental tool for studying eigenvalue and eigenfunction problems of the Laplacian. For example, the min-max principle is a crucial ingredient for the proof of Courant's nodal domain theorem for eigenfuctions of the Laplacian, see e.g. \cite{Cha} and \cite{SY94}. As an immediate consequence of the Courant's nodal domain theorem, the first eigenvalue of the Laplacian is simple.

The key to establish the min-max principle for an eigenvalue problem is to choose a suitable space of admissible functions. According to the estimate for $L_{A}$ in Theorem $\ref{Semi-boundedness}$, it turns out that the weighted Sobolev space $H^{1}(M)$ is the space of admissible functions for establishing the min-max principle for eigenvalues of the operator $L_{A}$ obtained in Theorem $\ref{SpectrumProperty}$.

Define the Dirichlet energy for the operator $L_{A}$ as the symmetric bilinear form
\begin{equation}
D(\phi, \psi)=\int_{M}\langle\nabla\phi, \nabla\psi\rangle d\vol_{g}+\int_{M}(R+A)\phi\psi d\vol_{g},
\end{equation}
for all $\phi$ and $\psi$ in $C^{\infty}_{0}(M)$.

By definition, it is easy to see that
\begin{equation}\label{DirichletEnergyEstimate}
|D(\phi, \psi)|\leq C(g, A)\|\phi\|_{H^{1}(M)}\|\psi\|_{H^{1}(M)},
\end{equation}
for any $\phi$ and $\psi$ in $C^{\infty}_{0}(M)$, and some constant $C(g, A)$ only depending on the metric $g$ and constant $A$. Thus, the symmetric bilinear form $D$ can be extended to $H^{1}(M)$, and we still use the notation $D(\phi, \psi)$ for $\phi$ and $\psi$ in $H^{1}(M)$.

Clearly,
\begin{equation}\label{AdmissibleEquality}
D(\phi, \psi)=(L_{A}\phi, \psi)_{L^{2}(M)},
\end{equation}
for all $\phi$ and $\phi$ in $C^{\infty}_{0}(M)$. Then, according to the estimate $(\ref{DirichletEnergyEstimate})$, the equality (\ref{AdmissibleEquality}) holds for all $\phi\in C^{\infty}_{0}(M)$ and $\psi\in H^{1}(M)$. And we can extend the equality $(\ref{AdmissibleEquality})$ further as the following.
\begin{lem}\label{AdmissibleEqualityLemma}
\begin{equation}
D(\phi, \psi)=(\widetilde{L_{A}}\phi, \psi)_{L^{2}(M)},
\end{equation}
for all $\phi\in Dom(\widetilde{L_{A}})$ and $\psi\in H^{1}(M)$, where $\widetilde{L_{A}}$ is the Friderichs extension of $L_{A}$.
\end{lem}
\begin{proof}
According to the construction of the Friedrichs extension $\widetilde{L_{A}}$ (see, e.g., the proof of the Neumann theorem in \cite{EK}),
\begin{eqnarray*}
Dom(\widetilde{L_{A}}) &\equiv& \overline{C^{\infty}_{0}(M)}\cap Dom(L^{*}_{A})\subset H^{1}(M)\\
\widetilde{L_{A}} &\equiv& L^{*}_{A}|_{Dom(\widetilde{L_{A}})},
\end{eqnarray*}
where $\overline{C^{\infty}_{0}(M)}$ is the completion of $C^{\infty}_{0}(M)$ with respect to the norm defined by $\|\phi\|^{2}_{L_{A}}=(L_{A}\phi, \phi)_{L^{2}(M)}$.

For any $\phi\in Dom(\widetilde{L_{A}})\subset H^{1}(M)$, and $\psi\in C^{\infty}_{0}(M)$,
\begin{equation}
(\widetilde{L_{A}}\phi, \psi)_{L^{2}(M)}=(L^{*}_{A}\phi, \psi)_{L^{2}(M)}=(\phi, L_{A}\psi)_{L^{2}(M)}=D(\phi, \psi).
\end{equation}
Then, by combining with the estimate $(\ref{DirichletEnergyEstimate})$, we complete the proof.
\end{proof}

Lemma $\ref{AdmissibleEqualityLemma}$ tells us that $H^{1}(M)$ is the space of admissible functions for the eigenvalue problem of $L_{A}$, and enables us to establish the following min-max principle for eigenvalues of $L_{A}$.
\begin{thm}\label{Min-maxPrinciple}
Let $0<\lambda_{1}\leq\lambda_{2}\leq\cdots$, and $\lambda_{k}\rightarrow+\infty$ as $k\rightarrow+\infty$, be eigenvalues of $L_{A}$. We can choose eigenfuctions $u_{i}$ such that $L_{A}u_{i}=\lambda_{i}u_{i}$ and $\{u_{1}, u_{2}, u_{3}, \cdots\}$ form a complete orthonormal basis of $L^{2}(M)$. Then
\begin{eqnarray*}
\lambda_{1} &=& \inf\{\frac{D(u,u)}{\|u\|_{L^{2}(M)}} \mid u\in H^{1}(M)\},\\
\lambda_{i} &=& \inf\{\frac{D(u,u)}{\|u\|_{L^{2}(M)}} \mid u\in H^{1}(M), \ \  (u, u_{j})_{L^{2}(M)}=0, \ \ j=1, \cdots, i-1\},
\end{eqnarray*}
for all $i=2, 3, 4, \cdots$.

Moreover, a function $u\in H^{1}(M)$ realizes the above infimum for $\lambda_{i}$ if and only if $u$ is an eigenfunction of $L_{A}$ with the eigenvalue $\lambda_{i}$.
\end{thm}

Recall, the nodal set of a continuous function $u: M^{n}\rightarrow\mathbb{R}$ is the set $u^{-1}(0)$, and a nodal domain of $u$ is a connected component of $\overline{M}\setminus u^{-1}(0)$.

In general, the nodal sets could be very strange. However, S. Y. Cheng's result for local behavior of nodal sets says that the nodal set of a solution to a elliptic equation $(\Delta+h(x))u=0$, where $h\in C^{\infty}(M^{n})$, is a $(n-1)$-dimensional smooth manifold after removing a lower dimensional closed set (see, e.g. Theorem 6.1 in Chapter 3 of \cite{SY94}).

By Theorem $\ref{Min-maxPrinciple}$ and S. Y. Cheng's result, and applying the argument in the proof of Theorem 6.2 in Chapter 3 of \cite{SY94}, we can obtain the following Courant's nodal domain theorem for eigenfunction of $L_{A}$.
\begin{thm}
Let $\lambda_{i}$ be the $i$-th eigenvalue of $L_{A}$, and $u_{i}$ be a corresponding eigenfuction. Then the number of nodal domains of $u_{i}$ is less than or equal to $i$.
\end{thm}
\begin{cor}
The first eigenvalue of $L_{A}$ is simple.
\end{cor}
\end{section}

\begin{section}{Asymptotic behavior of eigenfunctions of $-4\Delta+R$ on compact manifolds with a single cone-like singularity}

\noindent In this section, we obtain an asymptotic expansion for eigenfunctions of $-4\Delta+R$ near the singularity on manifolds with a single cone-like singularity.  The precise geometric structure near the singularity enables us to explicitly express eigenfunction in terms of some hypergeometric functions, and eigenvalues and eigenfunctions of the operator $-4\Delta+R$ on the cross section.

Let $(M^{n},g,p)$ be a compact Riemannian manifold with cone-like singularity $p$, and $U_{p}$ be a neighborhood of $p$ so that $U_{p}\backslash \{p\}$ is diffeomorphic to $(0,\varepsilon)\times N$, and on $U_{p}\backslash \{p\}$, $g=dr^{2}+r^{2}h_{0}$. Let $\mu_{1}<\mu_{2}\leq\mu_{3}\leq\cdots$ be eigenvalues of the operator $-4\Delta_{h_{0}}+R_{h_{0}}$ on the Riemannian manifold $(N,h_{0})$, which is the cross section of the conical neighborhood $U_{p}$ of $p$, and $\psi_{1},\psi_{2},\psi_{3},\cdots$ be corresponding normalized eigenfunctions, i.e.  $\|\psi_{i}\|_{L^{2}}=1$.

By using the classical Sobolev embedding theorem and elliptic regularity, with $s=2n$, we have
\begin{align*}
\|\psi_{i}\|_{L^{\infty}}
&\leq K_{s}\|\psi_{i}\|_{W^{s, 2}}\\
&\leq K_{s}C_{s}(\|\psi_{i}\|_{W^{s-2, 2}}+\|(-4\Delta_{h_{0}}+R_{h_{0}})\psi_{i}\|_{W^{s-2, 2}})\\
&= K_{s}C_{s}(1+|\mu_{i}|)\|\psi_{i}\|_{W^{s-2, 2}}\\
&\cdots\\
&\leq C(1+|\mu_{i}|)^{n}\|\psi_{i}\|_{W^{0, 2}}\\
&=C(1+|\mu_{i}|)^{n}.
\end{align*}
Let $u$ be an eigenfunction of the operator $-4\Delta_{g}+R_{g}$ with eigenvalue $\lambda$, i.e.
\begin{equation}\label{eigenvalue_equation}
-4\Delta_{g}u+R_{g}u=\lambda u.
\end{equation}
On the conical neighborhood $U_{p}\backslash\{p\}$, we do the the following expansion for the eigenfunction $u$:
\begin{equation}\label{eigenfunction_expansion}
u=\sum^{+\infty}_{i=1}u_{i}(r)\psi_{i}(x),
\end{equation}
where $x$ is the coordinate on $N$.

By plugging $(\ref{Lapace_on_cone})$, $(\ref{Scale_curvature_on_cone})$, and $(\ref{eigenfunction_expansion})$ in the equation $(\ref{eigenvalue_equation})$, we have
\begin{align*}
-4\Delta_{g}u+R_{g}u
&=\sum^{+\infty}_{i=1}(-4u^{''}_{i}\psi_{i}-4\frac{n-1}{r}u^{'}_{i}\psi_{i}-4\frac{1}{r^{2}}u_{i}\Delta_{g_{0}}\psi_{i})\\
&\ \ \ +\sum^{+\infty}_{i=1}(\frac{1}{r^{2}}R_{g_{0}}u_{i}\psi_{i}-(n-1)(n-2)\frac{1}{r^{2}}u_{i}\psi_{i})\\
&=\sum^{+\infty}_{i=1}[-4u^{''}_{i}-4\frac{n-1}{r}u^{'}_{i}-\frac{1}{r^{2}}(-\mu_{i}+(n-1)(n-2))u_{i}]\psi_{i}.
\end{align*}
Thus the eigenvalue equation for $u$ translates into the following ODE for each $u_{i}(r)$,
\begin{equation}
-4u^{''}_{i}-4\frac{n-1}{r}u^{'}_{i}-\frac{1}{r^{2}}(-\mu_{i}+(n-1)(n-2))u_{i}=\lambda u_{i}.
\end{equation}
In other words,
\begin{equation}\label{general_Bessel_equation}
u^{''}_{i}+\frac{n-1}{r}u^{'}_{i}+\frac{1}{4}(\lambda-\frac{1}{r^{2}}(\mu_{i}-(n-1)(n-2)))u_{i}=0.
\end{equation}

Now we solve the equation $(\ref{general_Bessel_equation})$ in three cases according to the sign of $\lambda$.

{\em Case 1:} $\lambda=0$. Then equation $(\ref{general_Bessel_equation})$ becomes the following Euler equation.
      \begin{equation}\label{Euler_equation}
      u^{''}_{i}+\frac{n-1}{r}u^{'}_{i}-\frac{1}{4}\frac{1}{r^{2}}[\mu_{i}-(n-1)(n-2)]u_{i}=0
      \end{equation}
     The general solutions are
      \begin{equation}\label{solution for 0 eigenvalue}
      u_{i}(r)=A_{i}r^{-\frac{n-2}{2}+\frac{\sqrt{\mu_{i}-(n-2)}}{2}}+B_{i}r^{-\frac{n-2}{2}-\frac{\sqrt{\mu_{i}-(n-2)}}{2}},
      \end{equation}
      where $A_{i}$ and $B_{i}$ are some constants. From the $L^2$ condition $u=\sum\limits^{+\infty}_{i=0}u_{i}\psi_{i}(x)\in L^{2}(M,g)$,  $u_{i}(r)=A_{i}r^{-\frac{n-2}{2}+\frac{\sqrt{\mu_{i}-(n-2)}}{2}}$ for large $i$, i.e.  $B_{i}=0$ for large $i$.

{\em Case 2:} $\lambda>0$. In this case, the equation transforms   to a Bessel equation.

      Let $$h_{i}(r)=(\frac{\sqrt{\lambda}}{2})^{-\frac{n-2}{2}}r^{\frac{n-2}{2}}u_{i}(\frac{2r}{\sqrt{\lambda}}).$$
      Then
      $$h^{''}_{i}(\frac{\sqrt{\lambda}r}{2})+\frac{1}{\frac{\sqrt{\lambda}r}{2}}h^{'}_{i}(\frac{\sqrt{\lambda}r}{2})+[1-\frac{1}{\frac{\lambda r^{2}}{4}}\frac{1}{4}(\mu_{i}-(n-2))]h_{i}(\frac{\sqrt{\lambda}r}{2})=0.$$
      Thus
      $$h_{i}(\frac{\sqrt{\lambda}r}{2})=A_{i} J_{\frac{1}{2}\sqrt{\mu_{i}-(n-2)}}(\frac{\sqrt{\lambda}r}{2})+ B_{i} Y_{\frac{1}{2}\sqrt{\mu_{i}-(n-2)}}(\frac{\sqrt{\lambda}r}{2}),$$
      where $A_{i}$ and $B_{i}$ are some constants, and $J_{\nu}(z)$ and $Y_{\nu}(z)$ Bessel functions of first and second kind respectively.
      Hence we obtain
      \begin{equation}\label{solution for positive eigenvalues}
      u_{i}(r)=A_{i} r^{-\frac{n-2}{2}}J_{\frac{1}{2}\sqrt{\mu_{i}-(n-2)}}(\frac{\sqrt{\lambda}r}{2})+ B_{i}r^{-\frac{n-2}{2}}Y_{\frac{1}{2}\sqrt{\mu_{i}-(n-2)}}(\frac{\sqrt{\lambda}r}{2}).
      \end{equation}

      Bessel functions have the following asymptotic behavior. If $\nu\rightarrow+\infty$ through real values, with $z\neq0$ fixed, then
      $$J_{\nu}(z)\sim \frac{1}{\sqrt{2\pi\nu}}(\frac{ez}{2\nu})^{\nu},$$
      $$Y_{\nu}(z)\sim -\sqrt{\frac{2}{\pi\nu}}(\frac{ez}{2\nu})^{-\nu}.$$
      Thus as in Case 1, for large $i$, $u_{i}(r)=A_{i} r^{-\frac{n-2}{2}}J_{\frac{1}{2}\sqrt{\mu_{i}-(n-2)}}(\frac{\sqrt{\lambda}r}{2})$, i.e.  $B_{i}=0$.

{\em Case 3:} $\lambda<0$. We use the results in \cite{B53}.

      If $1+\sqrt{\mu_{i}-(n-2)}$ is not an integer, then
      \begin{align}\label{solution for negative eigenvalues1}
      u_{i}=& A_{i}r^{-\frac{n-1}{2}}(\sqrt{-\lambda}r)^{\frac{1+\sqrt{\mu_{i}-(n-2)}}{2}}e^{-\frac{\sqrt{-\lambda}r}{2}}
              \sum^{+\infty}_{k=0}\frac{(\frac{1+\sqrt{\mu_{i}-(n-2)}}{2})_{k}}{(1+\sqrt{\mu_{i}-(n-2)})_{k}}\frac{(\sqrt{-\lambda}r)^{k}}{k!}\\
             & +B_{i}r^{-\frac{n-1}{2}}(\sqrt{-\lambda}r)^{\frac{1-\sqrt{\mu_{i}-(n-2)}}{2}}e^{-\frac{\sqrt{-\lambda}r}{2}}
              \sum^{+\infty}_{k=0}\frac{(\frac{1-\sqrt{\mu_{i}-(n-2)}}{2})_{k}}{(1-\sqrt{\mu_{i}-(n-2)})_{k}}\frac{(\sqrt{-\lambda}r)^{k}}{k!},\notag
      \end{align}
      where $A_{i}$ and $B_{i}$ are some constants, and $(x)_{k}=x(x+1)\cdots(x+n-1)$.\\

      If $1+\sqrt{\mu_{i}-(n-2)}=1+m$ is a positive integer, then
      \begin{align}\label{solution for negative eigenvalues2}
      u_{i}=& A_{i}r^{-\frac{n-1}{2}}(\sqrt{-\lambda}r)^{\frac{m+1}{2}}e^{-\frac{\sqrt{-\lambda}r}{2}}
            \sum^{+\infty}_{k=0}\frac{(\frac{1+m}{2})_{k}}{(1+m)_{k}}\frac{(\sqrt{-\lambda}r)^{k}}{k!}\\
            & +B_{i}r^{-\frac{n-1}{2}}(\sqrt{-\lambda}r)^{\frac{1+m}{2}}e^{-\frac{\sqrt{-\lambda}r}{2}}
            \frac{(-1)^{m-1}}{m!\Gamma(\frac{1-m}{2})}
            \{(\sum^{+\infty}_{k=0}\frac{(\frac{1+m}{2})_{k}}{(1+m)_{k}}\frac{(\sqrt{-\lambda}r)^{k}}{k!})\log(\sqrt{-\lambda}r)\notag\\
            & +\sum^{+\infty}_{k=0}\frac{(\frac{1+m}{2})_{k}}{(1+m)_{k}}(\psi(\frac{1+m}{2}+k)-\psi(1+k)-\psi(1+m+k))\frac{(\sqrt{-\lambda}r)^{k}}{k!}\notag\\
            & +\frac{(m-1)!}{\Gamma(\frac{1+m}{2})}\sum^{m-1}_{k=0}\frac{(\frac{1-m}{2})_{k}}{(1-m)_{k}}\frac{(\sqrt{-\lambda}r)^{k-m}}{k!}\},\notag
      \end{align}
      where $A_{i}$ and $B_{i}$ are some constants, and $\psi(x)$ is the logarithmic derivative of the Gamma function $\Gamma(x)$. So as in the previous cases, for the large $i$, $B_{i}=0$.

By the above explicit solutions for $u_{i}$ and estimates for eigenfunctions $\varphi_{i}$, we obtain the following asymptotic behavior for eigenfunction $u$.

\begin{thm}\label{AsymptoticExpansion}
Let $(M^{n},g,p)$ be a compact Riemannian manifold with a single cone-like singularity $p$ with $R_{h_{0}}>(n-2)$, and $u$ be an eigenfunction of the operator $-4\Delta_{g}+R_{g}$ on $M$. Then $u$ has an asymptotic expansion at the conical singularity $p$ as
$$u\sim \sum^{+\infty}_{j=1}\sum^{+\infty}_{l=0}\sum^{p_{j}}_{p=0}r^{s_{j}+l}(\ln r)^{p}u_{j,l,p},$$
where $u_{j,l,p}\in C^{\infty}(N^{n-1})$, $p_{j}=0$ or $1$, and $s_{j}=-\frac{n-2}{2}\pm\frac{\sqrt{\mu_{j}-(n-2)}}{2}$, where $\mu_{j}$ are eigenvalues of $-\Delta_{h_{0}}+R_{h_{0}}$ on $N^{n-1}$.
\end{thm}
\begin{proof}
On the conical part $U_{p}\setminus\{p\}$, 
$$u(r,x)=\sum^{\infty}_{i=1}u_{i}(r)\psi_{i}(x).$$

Set
$$\nu_{i}=\sqrt{\mu_{i}-(n-2)}.$$

{\em Case 1:}
$\lambda=0$. By the solution $(\ref{solution for 0 eigenvalue})$, there exists large enough $i_{0}\in\mathbb{N}$ such that for all $i\geq i_{0}$,
$$u_{i}(r)=A_{i}r^{-\frac{n-2}{2}+\frac{\nu_{i}}{2}},$$
and $\mu_{i}>(n-2)$.

For a fixed $r_{0}$, $u(r_{0},x)\in L^{2}(N)$, and by Parseval's identity
$$+\infty>\|u(r_{0},x)\|_{L^{2}(N)}=\sum^{\infty}_{i=1}|u_{i}(r_{0})|^{2}
\geq\sum^{\infty}_{i=i_{0}}|A_{i}|^{2}r_{0}^{-(n-2)+\nu_{i}}.$$
Then for all $r\leq\frac{r_{0}}{2}$,
\begin{align*}
\sum^{\infty}_{i=i_{0}}|u_{i}(r)\psi_{i}(x)|
& \leq C\sum^{\infty}_{i=i_{0}}|A_{i}|r^{-\frac{n-2}{2}+\frac{\nu_{i}}{2}}(1+|\mu_{i}|)^{n}\\
& = C\sum^{\infty}_{i=i_{0}}|A_{i}|r_{0}^{-\frac{n-2}{2}+\frac{\nu_{i}}{2}}(1+|\mu_{i}|)^{n}
    (\frac{r}{r_{0}})^{-\frac{n-2}{2}+\frac{\nu_{i}}{2}}\\
& \leq C(\sum^{\infty}_{i=i_{0}}|A_{i}|^{2}r_{0}^{-(n-2)+\nu_{i}})^{\frac{1}{2}}(\sum^{\infty}_{i=i_{0}}(1+|\mu_{i}|)^{2n}
    (\frac{1}{2})^{-(n-2)+\nu_{i}})^{\frac{1}{2}}\\
& <+\infty.
\end{align*}
Thus, $\sum\limits^{\infty}_{i=i_{0}}|u_{i}(r)\psi_{i}(x)|$ converges uniformly for $0\leq r\leq\frac{r_{0}}{2}$ and $x\in N$.

{\em Case 2:}
$\lambda>0$. By the solution $(\ref{solution for positive eigenvalues})$, there exists $i_{1}\in\mathbb{N}$ such that for all $i\geq i_{1}$,
$$u_{i}(r)=A_{i} r^{-\frac{n-2}{2}}J_{\frac{1}{2}\nu_{i}}(\frac{\sqrt{\lambda}r}{2})
=A_{i} r^{-\frac{n-2}{2}}(\frac{\sqrt{\lambda}r}{4})^{\frac{\nu_{i}}{2}}
\sum^{\infty}_{m=0}\frac{(\frac{\sqrt{\lambda}r}{4})^{2m}}{m!\Gamma(\frac{1}{2}\nu_{i}+m+1)}.$$
Fix $r_{0}>0$. Then for $r\leq r_{0}$ and $i>i_{1}$,
$$|A_{i}|r^{-\frac{n-2}{2}}\frac{1}{\Gamma(\frac{1}{2}\nu_{i}+1)}(\frac{\sqrt{\lambda}r}{4})^{\frac{\nu_{i}}{2}}<|u_{i}(r)|
<|A_{i}|r^{-\frac{n-2}{2}}\frac{C(r_{0})}{\Gamma(\frac{1}{2}\nu_{i}+1)}(\frac{\sqrt{\lambda}r}{4})^{\frac{\nu_{i}}{2}},$$
where $C(r_{0})=e^{\frac{\lambda r_{0}^{2}}{16}}$. Then
$$+\infty>\|u(r_{0},x)\|_{L^{2}(N)}=\sum^{\infty}_{i=0}|u_{i}(r_{0})|^{2}
\geq\sum^{\infty}_{i=i_{1}}|A_{i}|^{2}r_{0}^{-(n-2)}\frac{1}{(\Gamma(\frac{1}{2}\nu_{i}+1))^{2}}(\frac{\sqrt{\lambda}r_{0}}{4})^{\nu_{i}}.$$
Thus,
$$\sum^{\infty}_{i=i_{1}}|u_{i}(r)\varphi_{i}(x)|
\leq C(r_{0})C\sum^{\infty}_{i=i_{1}}|A_{i}|r^{-\frac{n-2}{2}}\frac{1}{\Gamma(\frac{1}{2}\nu_{i}+1)}(\frac{\sqrt{\lambda}r}{4})^{\frac{\nu_{i}}{2}}(1+|\mu_{i}|)^{n}
<+\infty,$$
again converges uniformly for $0\leq r\leq\frac{r_{0}}{2}$ and $x\in N$.

{\em Case 3:}
$\lambda<0$. By $(\ref{solution for negative eigenvalues1})$ and $(\ref{solution for negative eigenvalues2})$ there exists $i_{2}\in\mathbb{N}$ such that for all $i\geq i_{2}$
$$u_{i}=A_{i}r^{-\frac{n-1}{2}}(\sqrt{-\lambda}r)^{\frac{1+\nu_{i}}{2}}e^{-\frac{\sqrt{-\lambda}r}{2}}
\sum^{\infty}_{k=0}\frac{(\frac{1+\nu_{i}}{2})_{k}}{(1+\nu_{i})_{k}}\frac{(\sqrt{-\lambda}r)^{k}}{k!}.$$
Then for $r\leq\frac{r_{0}}{2}$ and $i>i_{2}$
$$|A_{i}|r^{-\frac{n-1}{2}}(\sqrt{-\lambda}r)^{\frac{1+\nu_{i}}{2}}\leq|u_{i}(r)|\leq e^{\frac{\sqrt{-\lambda}r_{0}}{2}}|A_{i}|r^{-\frac{n-1}{2}}(\sqrt{-\lambda}r)^{\frac{1+\nu_{i}}{2}}.$$
Thus, as above,
$$\sum^{\infty}_{i=i_{2}}|u_{i}(r)\psi_{i}(x)|\leq e^{\frac{\sqrt{-\lambda}r_{0}}{2}}C\sum^{\infty}_{i=i_{2}}|A_{i}|r^{-\frac{n-1}{2}}(\sqrt{-\lambda}r)^{\frac{1+\nu_{i}}{2}}(1+|\mu_{i}|)^{n}<+\infty,$$
converges uniformly for $0\leq r\leq\frac{r_{0}}{2}$ and $x\in N$.

Hence in all three cases, there exits $i_{0}\in\mathbb{N}$ such that $\sum\limits^{\infty}_{i=i_{0}}u_{i}(r)\psi_{i}(x)$ absolutely uniformly converges for all $r\leq\frac{r_{0}}{2}$ and $x\in N$. By plugging $(\ref{solution for 0 eigenvalue})$, $(\ref{solution for positive eigenvalues})$, $(\ref{solution for negative eigenvalues1})$ or $(\ref{solution for negative eigenvalues2})$ in $u(r,x)=\sum\limits^{\infty}_{i=1}u_{i}(r)\psi_{i}(x)$, we obtain the asymptotic expansion.

Similarly, we can show that derivatives of the expansion series also absolutely and uniformly converge. And the proof is then complete.
\end{proof}

One of the consequences of the asymptotic expansion in Theorem $\ref{AsymptoticExpansion}$, combining with the fact that eigenfunctions are in $H^{1}(M)$, is  the following.
\begin{cor}\label{AsymptoticOrder}
Let $(M^{n},g,p)$ be a compact Riemannian manifold with a single cone-like singularity $p$ with $R_{h_{0}}>(n-2)$. The eigenfunctions of $-4\Delta_{g}+R_{g}$ satisfy
$$u=o(r^{-\frac{n-2}{2}}), \ \ \ \text{as} \ \ r\rightarrow0.$$
\end{cor}
\end{section}


\begin{section}{Asymptotic behavior of eigenfunctions of $-4\Delta+R$ on compact manifolds with a single conical singularity}
\noindent In this section, we obtain an asymptotic order for eigenfunctions near the singularity on manifolds with a single conical singularity. For this purpose, by using a scaling technique, we first establish Sobolev inequality and elliptic estimate for weighted norms on a finite (excat) cone analogous to that on $\mathbb{R}^{n}$ in \cite{Bar86}.

We first work on a finite cone $(C_{\epsilon}(N)=(0,\epsilon)\times N, g=dr^{2}+r^{2}h_{0})$. Define weighted uniform $C^{k}_{\delta}$-norms on a finite cone $C_{\epsilon}(N)$ as
\begin{equation}\label{WeightedLpNorm}
\|u\|_{C^{k}_{\delta}(C_{\epsilon(N)})}=\sup_{C_{\epsilon}(N)}(\sum^{k}_{i=0}r^{i-\delta}|\nabla^{i}u|),
\end{equation}
for $k\in\mathbb{N}$ and $\delta\in\mathbb{R}$. When $k=0$, we use $C_{\delta}$ to denote $C^{0}_{\delta}$. Then similar to (\lowercase\expandafter{\romannumeral4}) of Theorem 1.2 in \cite{Bar86}, we use scaling technique to obtain the following weighted Sobolev inequality.
\begin{lem}\label{SobolevInequality}
If $u\in H^{k}_{\delta}(C_{\epsilon}(N))$, and $k>\frac{n}{2}+l$, then
\begin{equation}
\|u\|_{C^{l}_{\delta}(C_{\epsilon}(N))}\leq C\|u\|_{H^{k}_{\delta}(C_{\epsilon}(N))},
\end{equation}
for some constant $C=C(n,k,\delta,\epsilon)$.

Moreover,
$$|\nabla^{l}u(r,x)|=o(r^{-l+\delta}) \ \ \text{as} \ \ r\rightarrow0.$$
\end{lem}
\begin{proof}
Let $u(r,x)$ be a function on the finite cone $C_{\epsilon}(N)$, where $x$ is a coordinate on $N$, and set
\begin{equation}
u_{a}(r,x)=u(ar,x),
\end{equation}
for a positive constant $a$. And let $C_{r_{1},r_{2}}=(r_{1},r_{2})\times N$ be a annulus on the finite cone $C_{\epsilon}(N)$, for $r_{1}<r_{2}\leq\epsilon$. Then by a simple change of variables, we have
\begin{equation}\label{RescaledHNorm}
\|u\|_{H^{k}_{\delta}(C_{ar_{1},ar_{2}})}=a^{-\delta}\|u_{a}\|_{H^{k}_{\delta}(C_{r_{1},r_{2}})},
\end{equation}
and
\begin{equation}\label{RescaledCNorm}
\|u\|_{C^{l}_{\delta}(C_{ar_{1},ar_{2}})}=a^{-\delta}\|u_{a}\|_{C^{l}_{\delta}(C_{r_{1},r_{2}})}.
\end{equation}

Let $C_{j}=((\frac{1}{2})^{j+1}\epsilon,(\frac{1}{2})^{j}\epsilon)\times N$ be an annulus on the cone $C_{\epsilon}(N)$. For any fixed $j\in\mathbb{N}$, by choosing $a=(\frac{1}{2})^{j}$, $r_{1}=(\frac{1}{2})\epsilon$, and $r_{2}=\epsilon$ in $(\ref{RescaledHNorm})$ and $(\ref{RescaledCNorm})$, and using the usual Sobolev inequality, we have
\begin{align*}
\|u\|_{C^{l}_{\delta}(C_{j})} &=(\frac{1}{2})^{-j\delta}\|u_{(\frac{1}{2})^{j}}\|_{C^{l}_{\delta}(C_{0})}\\
                          &\leq (\frac{1}{2})^{-j\delta}C\|u_{(\frac{1}{2})^{j}}\|_{H^{k}_{\delta}(C_{0})}\\
                          &=C\|u\|_{H^{k}_{\delta}(C_{j})}\\
                          &\leq C\|u\|_{H^{k}_{\delta}(C_{\epsilon}(N))},
\end{align*}
where the constant $C$ is independent of $j$. Therefore, we obtain the Sobolev inequality
\begin{equation*}
\|u\|_{C_{\delta}(C_{\epsilon}(N))}\leq C\|u\|_{H^{k}_{\delta}(C_{\epsilon}(N))}.
\end{equation*}
Because $\|u\|_{H^{k}_{\delta}(C_{\epsilon}((N))}<\infty$ we have $\|u\|_{H^{k}_{\delta}(C_{j})}=o(1)$ as $j\rightarrow\infty$. Therefore, we have $|\nabla^{l}u(r,x)|=o(r^{-l+\delta})$ as $r\rightarrow0$, since $\sup\limits_{(\frac{1}{2})^{j+1}\epsilon<r<(\frac{1}{2})^{j}\epsilon}r^{l-\delta}|\nabla^{l}u(r,x)|\leq\|u\|_{C^{l}_{\delta}(C_{j})}\leq C\|u\|_{H^{k}_{\delta}(C_{j})}$.
\end{proof}

Similar to Proposition 1.6 in \cite{Bar86}, we also have the following elliptic estimate.
\begin{lem}\label{EllipticEstimate}
If $u\in H^{k-2}_{\delta}(C_{\epsilon}(N))$, and $Lu\in H^{k-2}_{\delta-2}(C_{\epsilon}(N))$, then
\begin{equation*}
\|u\|_{H^{k}_{\delta}(C_{\epsilon}(N))}\leq C(\|Lu\|_{H^{k-2}_{\delta-2}(C_{\epsilon}(N))}+\|u\|_{H^{k-2}_{\delta}(C_{\epsilon}(N))}),
\end{equation*}
for some constant $C=C(n,k,\delta,\epsilon)$.
\end{lem}
\begin{proof}
The inequality follows from the usual interior elliptic estimates and the scaling technique as in the proof of Lemma $\ref{SobolevInequality}$.
\end{proof}

Now we consider a finite asymptotic cone $(C_{\epsilon}(N)=(0,\epsilon)\times N, g=dr^{2}+r^{2}h_{r})$, where $h_{r}$ is a family of Riemannian metrics on $N$ satisfying $h_{r}=h_{0}+o(r^{\alpha})$ as $r\rightarrow0$ for some $\alpha>0$ and a Riemannian metric $h_{0}$ on $N$. On the finite asymptotic cone, we can also define weighted Sobolev norms and weighted uniform $C^{k}$-norms similar to $(\ref{WSN})$ and $(\ref{WeightedLpNorm})$. We use $\|\cdot\|_{\widetilde{H}^{k}_{\delta}(C_{\epsilon}(N))}$ and $\|\cdot\|_{\widetilde{C}^{k}_{\delta}(C_{\epsilon}(N))}$ to denote weighted norms on the finite asymptotic cone.

We make an extra assumption for the asymptotically conical metric $g=dr^{2}+r^{2}h_{r}$:
\begin{equation}\label{AsymptoticOfMetric}
|\nabla^{i+1}(h_{r}-h_{0})|\in C_{-i}(C_{\epsilon}(N)), \ \ \textit{for} \ \ 0\leq i\leq \frac{n}{2}+2,
\end{equation}
where the covariant derivative $\nabla$ and the norm $| \cdot |$ of tensors are with respect to the exactly conical metric $dr^{2}+r^{2}h_{0}$. Then asymptotic condition $(\ref{AsymptoticOfMetric})$ of the metric implies that $r^{i}|\nabla^{i}\omega|$ is bounded for all $0\leq i\leq \frac{n}{2}+2$, where $\omega$ is the difference tensor between the Levi-Civita connection for the asymptotically conical metric and the one for the exactly conical metric. And then as arguments in the proof of Theorem $\ref{compactness on manifolds}$, for sufficiently small $\epsilon$, these weighted norms with respect to the asymptotically conical metric on $C_{\epsilon}(N)$ are equivalent to corresponding weighted norms with respect to the exact cone metric on $C_{\epsilon}(N)$. Therefore, by Lemma $\ref{SobolevInequality}$ and Lemma $\ref{EllipticEstimate}$, we have the following Sobolev inequality and elliptic estimates on a sufficiently small finite asymptotic cone.
\begin{lem}\label{AsymptoticSobolevInequality}
If $\epsilon$ is sufficiently small, $u\in \widetilde{H}^{k}_{\delta}(C_{\epsilon}(N))$, and $k>\frac{n}{2}+1$, then
\begin{equation}
\|u\|_{\widetilde{C}^{l}_{\delta}(C_{\epsilon}(N))}\leq C\|u\|_{\widetilde{H}^{k}_{\delta}(C_{\epsilon}(N))},
\end{equation}
for $l=0,$ and $1$, and some constant $C=C(n,k,\delta,\epsilon)$.
\end{lem}
\begin{lem}\label{AsymptoticEllipticEstimate}
If $\epsilon$ is sufficiently small, $u\in \widetilde{H}^{k-2}_{\delta}(C_{\epsilon}(N))$, and $Lu\in \widetilde{H}^{k-2}_{\delta-2}(C_{\epsilon}(N))$, then
\begin{equation*}
\|u\|_{\widetilde{H}^{k}_{\delta}(C_{\epsilon}(N))}\leq C(\|Lu\|_{\widetilde{H}^{k-2}_{\delta-2}(C_{\epsilon}(N))}+\|u\|_{\widetilde{H}^{k-2}_{\delta}(C_{\epsilon}(N))}),
\end{equation*}
for $2\leq k\leq \frac{n}{2}+2$, and some constant $C=C(n,\delta,\epsilon)$, where $L$ is also the operator $-4\Delta+R$ with respect to the asymptotically conical metric.
\end{lem}

These Sobolev inequality and elliptic estimates imply the following asymptotic order estimate for eigenfunctions of $-4\Delta+R$ near the tip of a finite asymptotic cone.
\begin{thm}\label{GeneralAsymptoticOrderOnCone}
Let $u$ be an eigenfunction of $L=-4\Delta+R$ on a finite asymptotic cone $(C_{\epsilon}(N), dr^{2}+r^{2}h_{r})$ with $R_{h_{0}}>(n-2)$ and $(\ref{AsymptoticOfMetric})$. Then
$$|\nabla^{i}u|=o(r^{-\frac{n-2}{2}-i}), \ \ \ \text{as} \ \ r\rightarrow0,$$
for $i=0$ and $1$.
\end{thm}
\begin{proof}
As we only consider the asymptotic behavior of the eigenfunction near the tip of the cone, without lose of generality, we can assume $\epsilon$ is sufficiently small so that Lemma $\ref{AsymptoticSobolevInequality}$ and Lemma $\ref{AsymptoticEllipticEstimate}$ hold on $C_{\epsilon}(N)$. In the proof of Theorem $\ref{spectrum on cones}$, we have obtained that the eigenfunction $u\in \widetilde{H}^{1}(C_{\epsilon}(N))\equiv\widetilde{H}^{1}_{1-\frac{n}{2}}(C_{\epsilon}(N))$. Then $Lu\in \widetilde{H}^{1}_{1-\frac{n}{2}}(C_{\epsilon}(N))\subset \widetilde{H}^{1}_{1-2-\frac{n}{2}}(C_{\epsilon}(N))$, since $Lu$ is a scale multiple of $u$. Then by Lemma $\ref{AsymptoticEllipticEstimate}$, $u\in \widetilde{H}^{3}_{1-\frac{n}{2}}(C_{\epsilon}(N))$. By applying the elliptic bootstrapping, we obtain that $u\in \widetilde{H}^{[\frac{n}{2}]+2}_{1-\frac{n}{2}}(C_{\epsilon}(N))$. Therefore, by Lemma $\ref{AsymptoticSobolevInequality}$, $u=o(r^{-\frac{n-2}{2}})$, and $|\nabla u|=o(r^{-\frac{n-2}{2}-1})$, as $r\rightarrow0$.
\end{proof}

As a direct consequence of Theorem $\ref{GeneralAsymptoticOrderOnCone}$, eigenfunctions of $-4\Delta+R$ on a manifold with a single conical singularity have an asymptotic behavior near the singularity.
\begin{cor}\label{GeneralAsymptoticOrder}
Let $(M^{n},g,p)$ be a compact Riemannian manifold with a single conical singularity $p$ with $R_{h_{0}}>(n-2)$ and $(\ref{AsymptoticOfMetric})$ near the singularity $p$. The eigenfunctions of $-4\Delta_{g}+R_{g}$ on satisfy
$$|\nabla^{i}u|=o(r^{-\frac{n-2}{2}-i}), \ \ \ \text{as} \ \ r\rightarrow0,$$
for $i=0$ and $1$.
\end{cor}
\end{section}


\begin{section}{$\lambda$-functional on manifolds with a single conical singularity}

\noindent In this section, we define the Perelman's $\lambda$-functional on manifolds with a single conical singularity and obtain its first and second variation formulae as an application of spectral properties of the operator $-4\Delta+R$  obtained in the previous sections. And as an application of the first variation formula, we obtain that the $\lambda$-functional is monotonic increasing under the Ricci flow preserving the conical singularity.

Let $(M^{n},g,p)$ be a compact Riemannian manifold with a single conical singularity at $p$ with $R_{h_{0}}>(n-2)$ and $(\ref{AsymptoticOfMetric})$ near $p$. We define the $\lambda$-functional as the first eigenvalue of $-4\Delta+R$. By the results on the asymptotic behavior of the eigenfunctions, we also has 
$$\lambda(g)=\inf\{\ \mathcal{F}(g,u) \ \ | \ \ \int_{M}u^{2}dV_{g}=1, u\geq0\,  \ \ u\in H^{1}(M)\ \}.$$

 Let $u$ be the corresponding normalized positive eigenfunction, i.e. $\int_{M}u^{2}d\vol_{g}=1$ and
\begin{equation}\label{Eigenequation}
-4\Delta u+Ru=\lambda u.
\end{equation}
Let $u=e^{-\frac{f}{2}}$, then $(\ref{Eigenequation})$ becomes
\begin{equation}\label{EigenequationNew}
\lambda=2\Delta f-|\nabla f|^{2}+R.
\end{equation}

Let $g(t)$ for $t\in(-\tau,\tau)$ be a smooth family of metrics on $M^{n}$ with a single conical singularity at $p$ satisfying $R_{h_{0}(t)}>(n-2)$, and $(\ref{AsymptoticOfMetric})$ near $p$ for all $g(t)$, and $g(0)=g$. Differentiating $(\ref{EigenequationNew})$ in $t$ gives
\begin{equation}\label{Variation1}
\dot{\lambda}=2\dot{\Delta}f+2\Delta\dot{f}+\dot{R}-\dot{(|\nabla f|^{2})},
\end{equation}
where ``upperdot" denotes the derivative with respect to $t$ at $t=0$. Multiplying the equation $(\ref{Variation1})$ by $e^{-f}$ and then integrating over $M^{n}$, we have
\begin{equation}\label{Variation2}
\dot{\lambda}=\int_{M}(2\dot{\Delta}f+2\Delta\dot{f}+\dot{R}-\dot{(|\nabla f|^{2})})e^{-f}d\vol_{g}.
\end{equation}

Let's look at the second term in the integral in $(\ref{Variation2})$.
\begin{align*}
\int_{M\setminus C_{\epsilon}(N)}2\Delta\dot{f}e^{-f}d\vol_{g}
&=\int_{\partial C_{\epsilon}(N)}(\partial_{r}\dot{f})e^{-f}r^{n-1}d\vol_{h_{\epsilon}}\\
&\ \ \ -\int_{M\setminus C_{\epsilon}(N)}2\langle\nabla\dot{f},\nabla f\rangle e^{-f}d\vol_{g}\\
&=o(\epsilon^{-(n-2)-1+(n-1)})\\
&\ \ \ -\int_{M\setminus C_{\epsilon}(N)}2\langle\nabla\dot{f},\nabla f\rangle e^{-f}d\vol_{g}\\
&\rightarrow -\int_{M}2\langle\nabla\dot{f},\nabla f\rangle e^{-f}d\vol_{g} \ \ \text{as} \ \ \epsilon\rightarrow0,
\end{align*}
where the boundary term goes away in the limit because of the asymptotic behavior of the eigenfunction in Theorem $\ref{GeneralAsymptoticOrder}$. Here in the above we have used $(\partial_{r}\dot{f})e^{-f}=-2[(\partial_{r}\dot{u})u-(\partial_{r} u)\dot{u}]=o(r^{-(n-2)-1})$, as $r\rightarrow0$.

For the other terms, plug in the standard variation formulae for the scale curvature $R$ and the Laplacian $\Delta$ (see e.g. 1.174 and 1.184 in \cite{Bes87}, or \cite{DWW05}) into $(\ref{Variation2})$. Then similar to the second term, when we do integration by parts all boundary terms go away. Therefore, we obtain the same first variation formula as that on the smooth compact manifolds.
\begin{prop}\label{FirstVariation}
\begin{equation}\label{FVE}
\dot{\lambda}=\int_{M}\langle-Ric_{g}-Hess_{g}f,h\rangle_{g} e^{-f}d\vol_{g},
\end{equation}
where $h=\dot{g}$.
\end{prop}
\begin{cor}\label{CPOLF}
The critical points of $\lambda$-functional are Ricci-flat metrics with a single conical singularity at $p$.
\end{cor}
\begin{proof}
By Proposition $\ref{FirstVariation}$, a critical point is a metric $g$ with a single conical singularity at $p$ satisfying
$$-Ric_{g}-Hess_{g}f=0.$$

Now
\begin{align*}
\int_{M\setminus C_{\epsilon}(N)}\Delta(e^{-f})d\vol_{g}
&= \int_{\partial C_{\epsilon}(N)}\partial_{r}(e^{-f})r^{n-1}d\vol_{h_{\epsilon}}\\
&= o(r^{-(n-2)-1+(n-1)})\\
&= o(1) \rightarrow0 \ \ \text{as} \ \ r\rightarrow0,
\end{align*}
i.e. $\int_{M}\Delta(e^{-f})d\vol_{g}=0$. Therefore, the proof of Proposition 1.1.1 in \cite{CZ06} applies here.
\end{proof}

\begin{remark}
In the proof of Corollary $\ref{CPOLF}$, the asymptotic behavior of the function $e^{-f}$ near the singularity plays a crucial role. In general, in the conically singular compact case, the steady Ricci soliton equation $Ric_{g}+Hess_{g}f=$ does not imply that Ricci curvature vanishes if we do not impose any asymptotic assumption for potential function $f$ near the singularities. This is different from the smooth compact case.
\end{remark}

Another application will be the monotonicity under the Ricci flow. Here we consider the Ricci flow preserving the conical singularity. More precisely, $g(t)$, for $t\in [0, T)$, is a solution of Ricci flow equation $\frac{\partial}{\partial t}g=-2Ric_{g}$ on smooth manifold $M\setminus\{p\}$, and $(M, g(t), p)$ is a Riemannian manifold with a single conical singularity at $p$ for each $t\in [0, T)$. This kind of Ricci flow has been studied by Yin\cite{Yin10}\cite{Yin13} and Mazzeo, Rubinstein, and Sesum\cite{MRS15} for surfaces, and by Vertman \cite{Ver16} for higher dimensions case. Actually, in \cite{Ver16}, Vertman proved the short time existence for the Ricci flow within a class of incomplete edge manifolds, which includes the case of manifolds with isolated conical singularities.

As usual, twice the gradient flow of the $\lambda$-functional 
$$ \frac{\partial}{\partial t} g= -2(Ric_{g}+Hess_{g}f) $$
is equivalent to the Ricci flow. Therefore, we have the following monotoncity propety for the $\lambda$-functional $\lambda(g)$ under the Ricci flow preserving the conical singularities.

\begin{thm}\label{MonotonicityUnderRicciFlow}
	  $\lambda(g)$ is monotonically nondecreasing under the Ricci flow within the class of compact Riemannian manifolds with isolated conical singularities satisfying $R_{h_{0}}>(n-2)$ on each cross section and the asymptotic condition $(\ref{AsymptoticOfMetric})$ near each singular point. Moreover, the monotonicity is strict unless the flow is a steady gradient Ricci soliton with isolated conical singularities satisfying $R_{h_{0}}>(n-2)$ and the asymptotic condition $(\ref{AsymptoticOfMetric})$.
		\end{thm}

Then this monotonicity property directly implies a no steady or expanding breather result.

\begin{defn}
A Ricci flow $(M, g(t))$ preseving the conical singularities at $p$ is called a breather preserving conical singularities if for some $t_{1}<t_{2}$ and $\alpha>0$ the metrics $\alpha g(g_{1})$ and $g(t_{2})$ only differ by a diffeomorphism; the cases $\alpha=1$, $\alpha<1$, and $\alpha>1$ correspond to steady, shrinking, and expanding breathers preserving conical singularities, respectively.
\end{defn}

In \cite{Per02}, Perelman proved that a breather on smooth compact manifold is a gradient Ricci soliton. Now we can extend Perelman's no breather theorem in steady and expanding cases to compact manifolds with isolated conical singularities as follows.

\begin{cor}\label{NoBreathers}
\mbox{}\par
\begin{enumerate}
\item A steady breather preserving conical singularity within the class as in Theorem  $\ref{MonotonicityUnderRicciFlow}$ is necessarily a steady Ricci gradient soliton preserving conical singularity.\\
\item $\lambda(g)V^{\frac{2}{n}}(g)$ is monotonically nondecreasing under the Ricci flow within the classes as in Theorem  $\ref{MonotonicityUnderRicciFlow}$ whenever it is nonpositive. Moreover, the monotonicity is strict unless we are on a gradient Ricci soliton preserving conical singularities within the classes as in Theorem $\ref{MonotonicityUnderRicciFlow}$.\\
\item  An expanding breather preserving conical singularity within the class as in Theorem  $\ref{MonotonicityUnderRicciFlow}$ is necessarily an expanding gradient Ricci soliton preserving conical singularity.\\
\end{enumerate}
\end{cor}

Once we use the first variation formula in Proposition $\ref{FirstVariation}$, the proof of Corollary $\ref{NoBreathers}$ is the same as smooth compact case (see, \cite{Per02}).

Finally we end with the second variation which is important in studying the stability issue.
\begin{prop}\label{SVF}
At a critical point, i.e. a Ricci-flat metric $g$ with a single conical singularity, the second variation formula is given by
\begin{equation}
\ddot{\lambda}=\int_{M}\langle -\frac{1}{2}\Delta_{L,g}h+\delta_{g}^{*}\delta_{g} h+\frac{1}{2}Hess_{g}(\nu_{h}),h\rangle_{g} e^{-f}d\vol_{g},
\end{equation}
where $\Delta_{g}\nu_{h}=-\delta_{g}(\delta_{g} h)$.
\end{prop}
As we have seen in Corollary $\ref{CPOLF}$, at a critical point $g$ of the $\lambda$-functional, $Ric_{g}\equiv0$ and $f$ is a constant function. Thus for getting the second variation formula at a critical point, we only need to take derivative for the gradient, $-Ric-Hessf$, in the integral $(\ref{FVE})$, and then plug the variation fomula of Ricci curvature (see, e.g. 1.174 in \cite{Bes87}, or \cite{DWW05}) into the integrand. And the equation $(\ref{Variation1})$ and the variation formula of the scalar curvature $R$ imply the Possion equation characterization of $\nu_{h}=tr_{g}h-2\dot{f}$ given at the end of Proposition $\ref{SVF}$.
\end{section}




\begin{thebibliography}{XXXXX9}
\bibitem[Bar86]{Bar86} Bartnik, R.: The mass of an asymptotically flat manifold, Comm. Pure and Appl. Math. Vol. \uppercase\expandafter{\romannumeral39} 661-693, (1986)
\bibitem[Bat53]{B53} Bateman, H.,  Higher Transcendental Functions Volume I.  McGraw-Hill Book Company, INC. (1953)
\bibitem[Bes87]{Bes87} Besse, A. L.: Einstein manifolds, Spriger-Verlag, New York (1987)
\bibitem[BP03]{BP03} Botvinnik,  B., Preston, B.: Conformal Laplacian and Conical Singularities. Proceeding of the School on High-Dimensional Manifold Topology, ICTP, Trieste, Italy, World Scientific (2003), also  arXiv: math. DG/0201058 v.2 8 (2002)
\bibitem[BS87]{BS87} Br\"{u}ning, J., Seeley, R.: The resolvent expansion for second order regular singular operators. J. Functional Analysis \textbf{73}, 369-429 (1987)
\bibitem[Can81]{Can81} Cantor, M.: Elliptic operators and the decomposition of tensor fields. Bull. Amer. Math. Soc. \textbf{5}, 235-262 (1981)
\bibitem[CBC81]{CBC81} Choquet-Bruhat, Y., Christodoulou, D.: Elliptic systems in $H_{s,\delta}$ spaces on manifolds which are Euclidean at infinity. Acta Math. \textbf{146}, 129-150 (1981)
\bibitem[CDS15]{CDS15} Chen, X., Donaldson, S., Sun, S.: K\"ahler-Einstein metrics on Fano manifolds \uppercase\expandafter{\romannumeral1}, \uppercase\expandafter{\romannumeral2} and \uppercase\expandafter{\romannumeral3}, J. Amer. Math. Soc. \textbf{28}, 183-278 (2015)
\bibitem[Cha84]{Cha} Chavel, I.: Eigenvalues in Riemmanian Geometry. Academic Press, Inc. Orlando, Florida (1984)
\bibitem[CHI04]{CHI04} Cao, H.-D., Hamilton, R. S. and Ilmanen, T., Gaussian densities and stability for some Ricci solitons, arXiv:math.DG/0404165. (2004)
\bibitem[CSCB78]{CSCB78} Chaljub-Simonn, A., Choquet-Bruhat, Y.: Probl\`{e}mes elliptiques du second ordre sur une vari\'{e}t\'{e} euclidienne \`{a} l'infini. Ann. Fac. Sci. Toulouse \textbf{1}, 9-25 (1978)
\bibitem[CZ06]{CZ06} Cao, H-D., Zhu, X-P.: A complete proof of the Poincar\'{e} and geometrization conjectures-application of the Hamilton-Perelman theory of the Ricci flow. Asian J. Math. Vol. \textbf{10}(2), 165-492 (2006)
\bibitem[CZ12]{CZ12} Cao, H-D., Zhu, M.: On second variation of Perelman's Ricci shrinker entropy. Math. Ann. \textbf{353}, 747-763 (2012)
\bibitem[Che79]{Che} Cheeger, J.:  On the spectral geometry of spaces with cone-like singularities. Proc Natl Acad Sci U S A. \textbf{76}(5), 2103-2106 (1979)
\bibitem[DWW05]{DWW05} Dai, X., Wang, X., Wei, G.: On the stability of Riemannian manifold with parallel spinors. Invent. Math. \textbf{161}(1), 151-176 (2005)
\bibitem[EK96]{EK} Egorov, Y., Kondratiev, V.: On Spectral Theory of Elliptic Operators. Birkh\"auser (1996)
\bibitem[Ham82]{Ham82} Hamilton, R.: Three-manifolds with positive Ricci curvature. J. Differential Geom. \textbf{17}(2), 255-306 (1982)
\bibitem[Has12]{Has12} Haslhofer, R.: Perelman's lambda-functional and the stability of Ricci-flat metrics. Calc. Var. PDE \textbf{45}(3-4), 481-504 (2012)
\bibitem[HM14]{HM14} Haslhofer, R., M\"{u}ller, R.: Dynamical stability and instability of Ricci-flat metrics. Math. Ann. \textbf{360}(1-2), 547-553 (2014)
\bibitem[KMR97]{KMR97} Kozlov, V., Maz'ya, J., Rossmann, J.: Elliptic boundary value problems in domains with point singularities. Mathematical Surveys and Monographs, \textbf{52}. American Mathematical Society, Providence, RI (1997)
\bibitem[Kuf85]{Kuf85} Kufner, A.: Weighted Sobolev spaces. John Wiley \& Sons Limited. Germany (1985)
\bibitem[Loc81]{Loc81} Lockhart, R.: Fredholm properties of a class of elliptic operators on non-compact manifolds. Duke Math. J. \textbf{48}, 289-312 (1981)
\bibitem[LM85]{LM85} Lockhart, R., McOwen, R.: Elliptic differential operators on noncompact manifolds. Ann. Scuola Norm. Sup. Pisa. \textbf{12}(3), 409-447 (1985)
\bibitem[LP87]{LP87} Lee, J., Parker, T.: The Yamabe problem. Bull. Amer. Math. Soc. \textbf{17}(1), 833-866 (1987)
\bibitem[McO79]{McO79} McOwen, R.: Behavior of the Laplacian on weighted Sobolev spaces. Comm. Pure Appl. Math. \textbf{32}, 783-795 (1979)
\bibitem[MRS15]{MRS15} Mazzeo, R. Rubinstein, Y. A., Sesum, N.: Ricci flow on surfaces with conic singularities, Anal. PDE \textbf{8}(4), 839-882 (2015)
\bibitem[NW73]{NW73} Nirenberg, L., Walker, H.: The null spaces of elliptic partial differential operators in $\R^{n}$. J. Math. Anal. Appl. \textbf{42}, 271-301 (1973)
\bibitem[Ozu17]{Tri17} Ozuch, T.: Perelman's functionals on cones, Construction of type \uppercase\expandafter{\romannumeral3} Ricci flows coming out of cones. arXiv: 1707.06102v1 [math. DG] (2017)
\bibitem[Per02]{Per02} Perelman, G.: The entropy formula for the Ricci flow and its geometric applications. arXiv:math/0211159 (2002)
\bibitem[RS80]{RS1} Reed, M., Simon, B.: Methods of Modern Mathematical Physics \uppercase\expandafter{\romannumeral1}, Functional Analysis. Elsevier (Singapore) Pte Ltd. (1980)
\bibitem[RS75]{RS2} Reed, M., Simon, B.: Methods of Modern Mathematical Physics \uppercase\expandafter{\romannumeral2}, Fourier Analysis, Self-Adjointness. Elsevier (Singapore) Pte Ltd. (1975)
\bibitem[Ses06]{Ses06} Sesum, N.: Linear and dynamical stability of Ricci-flat metrics. Duke Math. J. \textbf{133}(1), 1-26 (2006)
\bibitem[SY94]{SY94} Schoen, R., Yau, S. T.: Lectures on differential geometry, International Press of Boston, Inc. (1994)
\bibitem[Tian15]{Tian15} Tian, G.: K-stability and K\"ahler-Einstein metrics. Comm. Pure Appl. Math. \textbf{68}(7), 1085-1283 (2015)
\bibitem[Tri78]{Tri78} Triebel, H.: Interpolation theory, function spaces, differential operators. North-Holland Mathematical Library, \textbf{18}. North-Holland Publishing Co., Amsterdam-New York (1978)

\bibitem[Tur00]{Tur00} Turesson, B. O.: Nonlinear potential theory and weighted Sobolev spaces. Springer-Verlag, Berlin, Heidelberg (2000)
\bibitem[Ver16]{Ver16} Vertman, B.: Ricci flow on singular manifolds, arXiv:1603.06545 (2016)
\bibitem[Yin10]{Yin10} Yin, H.: Ricci flow on surfaces with conical singularities, J. Geom. Anal. \textbf{20}(4), 970-995 (2010)
\bibitem[Yin13]{Yin13} Yin, H.: Ricci flow on surfaces with conical singularities \uppercase\expandafter{\romannumeral2}, arXiv:1305.4355 (2013)
\bibitem[Zha12]{Zha12} Zhang, Q. S.: Extremal of Log Sobolev inequality and W enptropy on noncompact manifolds, J. Funct. Anal. \textbf{263}, 2051-2101 (2012)

\end{thebibliography}
\end{document}